\documentclass{amsart}
\usepackage{amsmath}
\usepackage{paralist}
\usepackage{graphics} 
\usepackage{epsfig} 
\usepackage{graphicx}
\usepackage{epstopdf} 
\usepackage[colorlinks=true]{hyperref}
\hypersetup{urlcolor=blue, citecolor=red}

\usepackage{frenchineq}
\usepackage{algorithm,algorithmic}
\usepackage{hypergraph}
\usepackage{tikz}
\usetikzlibrary{shapes,arrows,decorations.markings,positioning}

\newtheorem{theorem}{Theorem}[section]
\newtheorem{corollary}{Corollary}[section]

\newtheorem{lemma}[theorem]{Lemma}
\newtheorem{proposition}{Proposition}[section]

\newtheorem*{problem}{Problem}
\newtheorem{assumption}{Assumption}
\theoremstyle{definition}
\newtheorem{definition}[theorem]{Definition}
\newtheorem{remark}{Remark}[section]
\newtheorem{example}{Example}[section]

\newcommand{\Reach}{\operatorname{Reach}}
\newcommand{\state}{S}

\newcommand{\pref}[1]{\eqref{#1}} 
\renewcommand{\emptyset}{\varnothing}
\newcommand{\R}{\mathbb{R}}

\newcommand{\Cpt}[1]{{S \setminus #1}}
\newcommand{\clo}[1]{{\bar{#1}}}
\newcommand{\size}{\mathsf{size}}
\newcommand{\MIN}{\textsf{MIN}}
\newcommand{\MAX}{\textsf{MAX}}
\newcommand{\tail}{\mathbf{t}}
\newcommand{\head}{\mathbf{h}}
\DeclareMathOperator{\argmin}{arg\,min}
\DeclareMathOperator{\argmax}{arg\,max}
\newcommand{\F}{\mathcal{F}}
\newcommand{\A}{\mathcal{A}}
\newcommand{\B}{\mathcal{B}}
\newcommand{\NonTrivialFP}{\textbf{NonTrivialFP}}
\newcommand{\Ergodicity}{\textbf{Ergodicity}}
\newcommand{\MonBool}{\textbf{MonBool}}
\newcommand{\IMinJMax}{\textbf{IMinJMax}}
\newcommand{\IisMin}{\textbf{I$=$Min}}
\newcommand{\colorpayment}{red}
\newcommand{\zero}{0}
\newcommand{\unit}{\mathbf{1}}
\newcommand{\indic}{\mathbf{1}}

\title[Ergodicity conditions for zero-sum games] 
      {Ergodicity conditions for zero-sum games}

\author[M.~Akian and S.~Gaubert and A.~Hochart]{}


 \email{marianne.akian@inria.fr}
 \email{stephane.gaubert@inria.fr}
 \email{hochart@cmap.polytechnique.fr}

\thanks{The last author is supported by a PhD fellowship of Fondation Math\'ematique Jacques Hadamard (FMJH).
The first two authors are partially supported by the PGMO Program of FMJH and EDF, and by ANR (MALTHY Project, number ANR-13-INSE-0003).}
\thanks{Initial version: May 15, 2014. Revised \today.}

\begin{document}
\maketitle

\centerline{\scshape Marianne Akian and St\'ephane Gaubert and Antoine Hochart}
\medskip
{\footnotesize
 \centerline{INRIA
  and CMAP, Ecole polytechnique, CNRS}
   \centerline{CMAP, Ecole polytechnique, Route de Saclay}
   \centerline{91128 Palaiseau Cedex, France}
} 

\bigskip

\begin{abstract}
A basic question for zero-sum repeated games consists in determining whether the mean payoff per time unit is independent of the initial state.
In the special case of ``zero-player'' games, i.e., of Markov chains equipped with additive functionals, the answer is provided by the mean ergodic theorem.
We generalize this result to repeated games. We show that the mean payoff is independent of the initial state for all state-dependent perturbations of the rewards if and only if an ergodicity condition is verified. The latter is characterized by the uniqueness modulo constants of nonlinear harmonic functions (fixed points of the recession function associated to the Shapley operator), or, in the special case of stochastic games with finite action spaces and perfect information, by a reachability condition involving conjugate subsets of states in directed hypergraphs.  We show that the ergodicity condition for games
only depends on the support of the transition probability,
and that it can be checked in polynomial time when the number of states is fixed. 
\end{abstract}

\section{Introduction}

\subsection{Motivation and related work}
\label{subsec-motivation}

The ergodicity of dynamical systems or of stochastic processes can be considered in several guises. In the elementary case of a discrete time Markov chain $(\xi_k)_{k\geq 0}$
with finite state space $S=[n]:=\{1,\dots,n\}$, ergodicity can be classically defined by any of the equivalent properties 
listed in the following theorem. Note that these properties only involve
the transition probability matrix
$P=(P(\xi_{k+1}=j\mid \xi_k=i))_{i,j=1,\ldots, n} \in \R^{n\times n}$.

\begin{theorem}
  \label{thm:ErgodicityMarkovChain}
  Let $P\in \R^{n\times n}$ be a stochastic matrix.
  The following properties are equivalent.
  \begin{enumerate}
    \item\label{it-1} Every vector $\eta\in \R^n$ such that $P\eta = \eta$ is constant;
   \item For every vector $g\in \R^n$, the Cesaro limit
   \begin{align}
     \lim_{k\to\infty} k^{-1}(g+Pg+\dots+P^{k-1}g)
     \label{e-cesaroold}
   \end{align}
   is a constant vector;
    \item\label{it-2} For every vector $g\in \R^n$, the ergodic equation
   \begin{align}
     g+Pu = \lambda \unit + u \enspace,\label{e-ergodicmc}
   \end{align}
   where $\unit$ denotes the unit vector of $\R^n$,
   admits a solution $(\lambda,u)\in \R\times \R^n$;
   \item\label{prop-ergodic-graph} The directed graph associated to the matrix $P$ has only one final class;
   \item\label{it-last} The matrix $P$ has only one invariant measure, that is 
a stochastic row vector $m \in \R^{1 \times n}$ such that $m P = m$.
 \end{enumerate}
\end{theorem}
Recall that a matrix $P=(P_{ij}) \in \R^{n\times n}$
(resp.\ a row vector $m=(m_j) \in \R^{1 \times n}$)
is said to be {\em stochastic} when all its entries are nonnegative
and each of its rows sums to one, meaning that 
$P_{ij}\geq 0$ and $\sum_{\ell=1}^n P_{i\ell}=1$ for all $i,j\in [n]$
(resp.\ $m_j\geq 0$ for all $j \in [n]$ and $\sum_{j=1}^n m_j=1$).
The {\em directed graph} associated with $P$ is composed of the nodes $1,\dots,n$ and of the arcs $(i,j)$, $i \in [n], \; j \in [n]$ with $P_{ij}>0$.
A {\em class} of the matrix $P$ is a maximal set of nodes such that every two nodes of the set are connected by a directed path. 
A class is said to be {\em final} if every path starting from a node of this class remains in this class. 
We refer the reader to~\cite[Chap.~8]{bermanplemmons} for details. The previous properties are well known, in particular, the equivalence between \pref{prop-ergodic-graph} and~\pref{it-last} follows from Theorem 3.23 in the latter reference, 
whereas the remaining equivalences follow from Theorem~6.1 in~\cite{whittle}.

The scalar $\lambda$ in the ergodic equation~\eqref{e-ergodicmc}, known
as the {\em ergodic constant}, gives the coordinates of the constant vector~\eqref{e-cesaroold}.

The term ergodicity is generally used to refer to the uniqueness
of the invariant measure, and so, following Kemeny and Snell~\cite{kemenysnell}, we call {\em ergodic} a Markov chain with the above properties
of its transition probability matrix.
We warn the reader that some authors use the word ``ergodic'' 
in a stronger sense, requiring, for a finite Markov chain,
the matrix $P$ to be irreducible and aperiodic.

In this paper, we extend the notion of ergodicity to zero-sum two-player repeated games with finite state space $S=[n]$. We refer the reader to Section~\ref{sec-games} for the detailed definition of these games. 
\if{The latter games can be defined as follows.
We assume that the {\em actions spaces} $A_i$ and $B_i$ in state $i\in S$
of players \MIN\ and \MAX, respectively, are given and nonempty.
Then,  a {\em transition payment} is a function $r: (i,a,b)\mapsto r_i^{ab}$,
from the 
``state-actions space'' $\cup_{i\in S}(\{i\}\times A_i \times B_i)$ to $\R$,
and a {\em transition probability} is a function 
 $P: (i,a,b)\mapsto P_i^{ab}$, from
the same space
to the set $\Delta(S)$ of probabilities over $S$, identified with the subset of $\R^{1 \times n}$ of 
stochastic row vectors.
We denote by  $\Gamma(r,P)$ the repeated game with 
transition payment $r$ and transition probability $P$. At each stage,
if the current state is $i$, player \MIN\ selects an action $a\in A_i$,
player \MAX\ subsequently selects an action $b\in N_i$, player \MIN\ pays
$r_i^{ab}$ to player \MAX, and the probability that $j\in S$ be the next
state is given by $(P_i^{ab})_j$. We assume that the information
is perfect, so that each player observes the state and the previous
actions of the other player. }\fi 
For the moment, we shall only need to know that the game in horizon $k$ with initial state $i$ has a value, denoted by $v^k_i\in \R$, and that the value
vector $v^k=(v^k_i)_{1\leq i \leq n}$ is determined
from the
{\em Shapley operator} $T=T(r,P)$. The latter
is the  map $\R^n\to \R^n$ given by 
\begin{equation}\label{defshapley}
  [T(r,P)(x)]_i = \inf_{a\in A_i}\sup_{b\in B_i} (r_i^{ab}+ P_i^{ab}x) \enspace ,
\end{equation}
for all $x = (x_i)_{i \in S}$. Here,
$A_i$ denotes the set of actions of player \MIN\ in state $i\in S$,
$B_{i}$ denotes the set of actions of player \MAX\ in the same
state, $r_i^{ab}$ denotes an running payment made
by player \MIN\ to player \MAX\ in state $i$ when the actions $a,b$
are chosen, and $P_{i}^{ab}$ is a row vector such that $(P_i^{ab})_{j}$
represents the probability of transition from state $i$ to state
$j$, when the actions $a,b$ are chosen. 
It is known that the value vector $v^k=(v^k_i)_{i\in\state}$
can be computed recursively by 
\[
v^k=T(v^{k-1}),\qquad v_0 =0 \enspace .
\]
Here, we will be interested in the {\em mean payoff vector}
\begin{align*}
  \chi(T):=\lim_{k \to +\infty} \frac{v^k}{k} = 
  \lim_{k\to+\infty} \frac{T^k(0)}{k}\enspace ,
\end{align*}
where $T^k:= T\circ\dots\circ T$ denotes the $k$th iterate of $T$,
so that $[\chi(T)]_i$ represents the mean payoff per time unit
of the game starting from state $i$, as the horizon tends to infinity.

The question of the existence of the mean payoff vector
has been studied by several authors, including Bewley,
Kohlberg, Mertens, Neyman, Rosenberg, Sorin, see~\cite{KN81,MertNeym81,RS01,NS03}, and also~\cite{Rena11,vigeral2013,BGV13} for some recent results.

A basic analytic tool to establish
the existence of the limit is the so called nonlinear {\em ergodic equation}
\begin{align}
  T(u)= \lambda\unit + u \enspace.\label{e-ergodic}
\end{align}
If a solution $(\lambda,u)\in \R\times \R^n$ exists, then, it it easily seen that 
\[
\chi(T) = \lambda \unit \enspace .
\]
In particular, the mean payoff is independent of the initial state,
and it is given by the ergodic constant $\lambda$, as in the case of Markov chains. The ergodic equation has been much studied in the one-player stochastic
case, i.e., in ``ergodic control'', where it is also
known as the ``average case optimality equation'', see~\cite{hernandezlasserrebook} for background.

The ergodic equation~\eqref{e-ergodic} is equivalent to a nonlinear spectral
problem which has also received attention in nonlinear Perron-Frobenius
theory, see specially the work of Nussbaum~\cite{nussbaum88,nussbaum89},
and also~\cite{GG04,LN12}. Indeed, the map $T$ is conjugate
to the self-map $G= \exp \circ T \circ \log$ of the interior
of the standard positive cone of $\R^n$, $C := \{x \in \R^n \mid x \geq 0 \}$, where 
$\exp$ is the map from $\R^n$ to the interior of $C$ which
does $\exp$ entrywise, and $\log:= \exp^{-1}$. 
The ergodic problem is equivalent to 
the nonlinear spectral problem 
\begin{align}
  G(v)=\mu v, \quad v\in \operatorname{int}C,
  \quad \mu>0\label{e-pf} \enspace .
\end{align}
Since the map $G$
is order-preserving and positively homogeneous of degree one,
conditions for the existence of an eigenpair $(v,\mu)$ may
be thought of as nonlinear extensions of the Perron-Frobenius theorem.
It is useful to keep this equivalence in mind as several results relevant to Problem~\eqref{e-ergodic} have appeared in the context of the nonlinear eigenproblem~\eqref{e-pf}, see for instance~\cite{GG04,CH10}.

The problem of characterizing the set of solutions $u$ of the ergodic equation $T(u) = \lambda \unit + u$ has also appeared in the setting of max-plus spectral theory~\cite{BCOQ92,AGW}, and in weak KAM theory~\cite{FS05,Fathi-book}. These theories concern the one-player deterministic case. 
It is known that the above set is sup-norm isometric to a set of Lipschitz functions on a certain set (critical classes in the max-plus setting, or projected Aubry set in the weak KAM setting). 
Some of these results have been extended to one-player stochastic games with finite state space in~\cite{AG03}.
The extension of such results to the two player case appears to be an open question, which is among the motivations leading to the present study.

A useful tool to address the issue of the solvability
of the ergodic equation~\eqref{e-ergodic}, or of the corresponding nonlinear eigenproblem~\eqref{e-pf},
is the {\em recession function} associated with the Shapley operator,
\begin{equation}\label{def-recess}
\hat{T}: x \in \R^n \mapsto \hat{T}(x)=\lim_{\rho \to +\infty} \frac{T(\rho x)}{\rho}
\enspace,
\end{equation}
which has already been used in several ways~\cite{RS01,Sor04,GG04}.
In particular, Rosenberg and Sorin~\cite{RS01} gave conditions for the existence of the mean payoff vector of a two-person zero-sum stochastic game.
In their framework, the recession function appears as the Shapley operator of the ``projective'' game, which corresponds to the game with no running payments.

If the transition payment $r$
is bounded, the recession function $\hat{T}$ does exist,
and it is given by
\begin{equation}
  \label{eq:PaymentFree}
  [\hat T(x)]_i = \inf_{a \in A_i} \sup_{b \in B_i} P_i^{a b} x, \quad i \in \state, x \in \R^n \enspace .
\end{equation}
Hence, $\hat{T}=T(0,P)$, with $T$ as in~\eqref{defshapley},
so that the recession function of the Shapley operator associated with the game with payment function $r$
is merely the Shapley operator of the game in which $r$ is replaced by $0$.
For this reason, we shall refer to the
maps of the form~\eqref{eq:PaymentFree} as {\em payment-free Shapley operators}. 

Observe that every constant vector is
a fixed point of a payment-free Shapley operator. We shall refer to such a fixed
point as {\em trivial}. 
In~\cite{GG04}, Gaubert and Gunawardena show that the ergodic equation is solvable if $\hat{T}$ has only trivial fixed points. 
A sufficient explicit condition for this to hold, involving
a sequence of aggregated directed graphs, generalizing the classical
directed graph of Perron-Frobenius theory, was given there.

Then, in~\cite{CH10}, Cavazos-Cadena and Hern{\'a}ndez-Hern{\'a}ndez
introduced a weak convexity property, 
and showed that when the conjugate map $G= \exp \circ T \circ \log$  is weakly convex,
the recession function 
$\hat{T}$ has only trivial fixed points
if and only if the first of the directed graphs of~\cite{GG04}
consists of a single final class and of trivial classes
(reduced to one node, and loop free). They deduced
that when $G$ is weakly convex,
the ergodic equation for all maps $g+T$ with $g\in \R^n$
is solvable if and only if $\hat{T}$ has only trivial fixed
points.
We shall consider the same additive perturbations $g+T$ of the Shapley operator, but without any assumption on $T$ except that the payment $r$ be bounded. 
Indeed, this weak convexity property
is rarely satisfied for games although 
it captures an interesting class of risk sensitive problems.

\subsection{Description of the main results}
\label{subsec-desc}

Our main results, summarized in Theorem~\ref{th-game-ergodicity} at the end of the paper, show that most of the classical characterizations of ergodicity
for finite state Markov chains, seen as zero-player games, carry over to the two-player case.
More precisely, given a zero-sum game with finite state space and bounded transition payment $r$, we show in Section~\ref{realizable-section} that the following conditions are equivalent:
\begin{enumerate}
  \item all the fixed points of the recession function $\hat T$ of the Shapley operator $T$ are trivial (i.e\ constant);
  \item all the games obtained by adding to the transition payment $r$ 
  a perturbation depending only of the state
  have a constant mean payoff vector;
  \item the ergodic equation~\eqref{e-ergodic} is solvable for all maps $g+T$ with $g \in \R^n$. 
\end{enumerate}
In the zero-player special case, the above conditions correspond to
Points~\pref{it-1}--\pref{it-2} of Theorem~\ref{thm:ErgodicityMarkovChain}.
Hence, a zero-sum game will be said to be {\em ergodic} if it satisfies one of these
properties. 
An ingredient of this equivalence is the result of
Gaubert and Gunawardena in~\cite{GG04} described above.

In Section~\ref{galois-section},  we give a characterization of ergodicity in terms of a Galois connection
acting on faces of the hypercube $[0,1]^n$.
Then, in Section~\ref{compact-section}, we show that under a compactness assumption on the action spaces and a continuity assumption on the transition probability, the latter characterization of ergodicity of a game (involving a Galois connection) is equivalent to a reachability condition involving a pair of directed hypergraphs.
These two characterizations
 are a fundamental discrepancy with Point~\pref{prop-ergodic-graph} of the zero-player case.
However, the characterization of ergodicity involving the hypergraphs
still keeps the same flavor.
Indeed, the condition that a directed graph
has only one final class can be thought of as an accessibility condition
in this directed graph (see in particular Remark~\ref{remark-hypergraph-graph}).
Under the compactness and continuity assumption on the action spaces and the transition probability, the Galois connection as well as the hypergraphs are shown to depend only on the {\em support} of the transition probability $P$,
which we define to be the set of points at which the function $(i,a,b,j) \mapsto (P_{i}^{ab})_j$ takes nonzero values.
As a result, we get that the ergodicity of a game is a structural property depending only on the support of the transition probability.

We then consider (in Section~\ref{algorithmic-section}) several algorithmic problems concerning games with finite action spaces.
The first one is to check ergodicity.
The restricted version of this problem concerning deterministic games was addressed by Yang and Zhao~\cite{YZ04}, in the context of discrete event systems.
They showed that this problem is coNP-hard. 
However we show, as a corollary of the hypergraph characterization, that checking the ergodicity of a stochastic game is fixed parameter tractable:
if the dimension is fixed, we can solve it in polynomial time.
Note also that ergodicity can be checked in polynomial time for one-player stochastic games~\cite{AG03}.
We finally characterize the situation in which there exists a fixed point having its minimal and maximal entries in prescribed positions. As a by product,
we get a polynomial time algorithm to check the latter
property, from which it follows
that checking ergodicity is a coNP-complete problem. 
In Section~\ref{ex-section}, we illustrate our results on some examples.

The present results have been announced in the conference article~\cite{hochartMTNS}.

\section{Zero-sum games with perfect information and mean-payoff}
\label{sec-games}

\subsection{Basic definitions and results}

\label{sec-games-basic}
In this subsection, we describe formally the zero-sum game with perfect
information mentioned above, and state preliminary results. 

Recall that $S=[n]$ is the state space, $A_i$ is the set of
actions of player \MIN, $B_i$ is the set of actions of player \MAX,
$(i,a,b) \mapsto r_i^{ab}$ from $\cup_{i\in S} (\{i\}\times A_i \times B_i)$
to $\R$ is the transition payment,
and $(i,a,b)\mapsto P_i^{ab}$ from the same set to
$\Delta(S) \subset \R^{1 \times n}$, the set of nonnegative row vectors of sum one,
is the transition probability.
This game,
which we denote by $\Gamma(r,P)$, 
is played as follows.
Starting from a given state $i_0$ at time $k=0$, known by the players, \MIN\ chooses an action $a_0 \in A_{i_0}$. Then, knowing this choice, player \MAX\ chooses an action $b_0 \in B_{i_0}$.
Player \MIN\ has to pay $r_{i_0}^{a_0 b_0}$ to player \MAX\ and the next state, $i_1$, is chosen according to the probability $P_{i_0}^{a_0 b_0}$.
The same procedure is repeated at each time step, giving an infinite sequence $(i_\ell, a_\ell, b_\ell)_{\ell \geq 0}$.

A strategy $\sigma$ (resp.\ $\tau$) of player \MIN\ (resp.\ \MAX) is a map which assigns an action of player \MIN\ (resp.\ \MAX) to every finite history known by the player.
A triple $(i_0,\sigma,\tau)$ defines a probability measure on the set of plays (or histories), that is, the set of sequences $(i_\ell, a_\ell, b_\ell)_{\ell \geq 0}$ for which $a_\ell \in A_{i_\ell}$ and $b_\ell \in B_{i_\ell}$.
We denote by $\mathbb E_{i_0,\sigma,\tau}$ the corresponding expectation.
The total payoff of the game with finite horizon $k$ (consisting in $k$ time steps, that is $k$ successive alternated moves of players \MIN\ and \MAX) is given by
\[
J^k_{i_0}(\sigma,\tau) = \mathbb E_{i_0,\sigma,\tau} \left[ \sum_{\ell=0}^{k-1} r_{i_\ell}^{a_\ell b_\ell} \right] \enspace .
\]
Player \MIN\ wishes to minimize this quantity, while player \MAX\ wishes to maximize it.
The value of the $k$-stage game  (the game played in finite horizon $k$)
starting at state $i$ is thus defined as
\[
v^k_{i} = \inf_\sigma \sup_\tau J^k_i(\sigma,\tau) \enspace ,
\]
the infimum and the supremum being taken over the set of strategies of players \MIN\ and \MAX, respectively. Here, the infimum and supremum commute.

It is known (see e.g.~\cite{NS03}) that the value vector $v^k=(v_i^k)$
satisfies $v^k=T(v^{k-1})$ and $v^0=0$, where $T=T(r,P)$ is the Shapley
operator defined by~\eqref{defshapley}.

Let $\A$ denote the set of (feedback) policies
of player \MIN, which are the maps $\sigma$
from $S$ to $\cup_{i\in S} A_i$ such that $\sigma(i)\in A_i$ for all $i\in S$,
and let $\B$ denote
the set of policies of player \MAX, which are the maps $\tau$ from 
$\cup_{i\in S} (\{i\}\times A_i)$ to $\cup_{i\in S} B_i$ such that 
$\tau(i,a)\in B_i$ for all $i\in S$ and  $a\in A_i$.
Recall that a strategy of player \MIN\ (resp.\ \MAX) is Markovian
if it only depends on the information of the current stage $k\geq 0$,
that is $a_k=\sigma_k(i_k)$ for some $\sigma_k\in\A$
(resp.\ $b_k=\tau_k(i_k,a_k)$  for some $\tau_k\in \B$).
Moreover, such a strategy is stationary if it is independent of $k$
($\sigma_k=\sigma\in\A$ and $\tau_k=\tau$ for all $k\geq 0$),
in which case it can be identified with the corresponding policy.
Then it is known that the above (dynamic programming) equation
provides optimal or $\epsilon$-optimal strategies of the two players 
that are Markovian.
Indeed, $T$ can be rewritten as follows:
\[ T(x) = \inf_{\sigma\in \A}\sup_{\tau\in \B} (r^{\sigma\tau}+ P^{\sigma\tau}x) 
=\sup_{\tau\in \B}  \inf_{\sigma\in \A} (r^{\sigma\tau}+ P^{\sigma\tau}x) \enspace ,\]
where 
$P^{\sigma\tau}_i= P_i^{\sigma(i)\tau(i,\sigma(i))}$,
and similarly for $r^{\sigma\tau}$, and
the infimum and supremum are taken for the usual partial order of $\R^n$
(the product partial order of the usual order on $\R$).
Moreover, the infimum and supremum can be approached arbitrarily
by the value of $r^{\sigma\tau}+ P^{\sigma\tau}x$ for some policies
$\sigma$ and $\tau$, and they are equal to such a value
when the action spaces $A_i$ and $B_i$ are compact and the transition payment
and probability functions are continuous.
In the latter case, we say that $\sigma$ and $\tau$ are optimal for $T(x)$.
Optimal strategies for the game in horizon $k\geq 0$ 
are then obtained by taking for all $0\leq \ell<k$,
$a_\ell=\sigma_{\ell}(i_\ell)$ and $b_\ell=\tau_{\ell}(i_\ell,a_\ell)$
for some $\sigma_\ell\in\A$ and $\tau_\ell\in \B$
optimal for $T(v^{k-\ell-1})$.

The Shapley operator $T$ satisfies the following properties:
\begin{itemize}
  \item[-] {\em monotonicity}, a.k.a.\ {\em order preservation}: $x \leq y \enspace \Rightarrow \enspace T(x) \leq T(y)$, where $\R^n$ is endowed with its usual partial order;
  \item[-] {\em additive homogeneity}: $T(x + \alpha \unit) = T(x) + \alpha \unit, \enspace x \in \R^n, \enspace \alpha \in \R$, recalling that $\unit$ denotes the unit vector of $\R^n$;
  \item[-] {\em nonexpansiveness} in the sup-norm: $\|T(x)-T(y)\| \leq \|x-y\|, \enspace x, y \in \R^n$, where $\|x\|:=\max_{1\leq i\leq n}|x_i|$.
\end{itemize}

\subsection{Games with mean payoff}

The {\em mean payoff vector} is defined as the limit 
\[ \chi(T)=\lim_{k\to\infty} \frac{T^k(x)}{k} \enspace , \]
for all $x\in \R^n$. Since $T$ is nonexpansive, the existence and the value of the latter limit is independent of the choice of $x$.
In particular, we have the following standard result:
\begin{proposition}\label{ergodic-solvable}
  If the following {\em ergodic equation} is solvable:
  \begin{equation}
    \label{eq:Ergodic}
    \exists (\lambda,u) \in \R \times \R^n, \quad T(u) = \lambda \unit + u \enspace ,
  \end{equation}
  then $\chi(T)$ exists and is equal to $ \lambda \unit$.
  In particular, the average payment (per time unit) of $\Gamma(r,P)$ is asymptotically independent of the initial state.
\end{proposition}
\begin{proof}
  Since $T$ is additively homogeneous, we have $T^k(u)=k\lambda \unit +u$,
  and so $\chi(T)=\lim_{k\to\infty}T^k(u)/k = \lambda \unit$.
\end{proof}
Moreover,  if $u$ is a solution of the above ergodic equation,
optimal policies $\sigma$ and $\tau$ 
of players \MIN\ and \MAX\ for $T(u)$, if they exist,
provide optimal strategies of the two players 
that are Markovian and stationary.

The ergodic equation~\eqref{eq:Ergodic} can be studied by means of the {\em recession} function $\hat{T}$ of $T$, defined by~\eqref{def-recess}.
The recession function of $T$ is well defined as soon as the transition payment is  bounded. Then, 
$\hat T$ is
given by~\eqref{eq:PaymentFree}, so that $\hat{T}=T(0,P)$.

\begin{definition}[Payment-free Shapley operators]
  A Shapley operator is said to be {\em payment-free}
  if it is of the form $F=T(0,P)$, 
  where  $P$ is a transition probability and $T$ is as in~\eqref{defshapley}.
\end{definition}

As any Shapley operator, a payment-free Shapley operator $F$ is monotone and additively homogeneous.
It is also positively homogeneous, that is, $F(\lambda x) = \lambda F(x)$,
for all $x \in \R^n, \lambda > 0$.
As a consequence, it satisfies $F(\lambda \unit) = \lambda \unit$ for every $\lambda \in \R$.
We call such fixed points the {\em trivial} fixed points of $F$. 
We shall use the following sufficient condition for the solvability
of the ergodic equation. 

\begin{theorem}[{Corollary of Gaubert and Gunawardena~\cite[Theorems~9 and 13]{GG04}}]
  \label{theo-GG}
  Consider a game $\Gamma(r,P)$, such that the recession function
  $\hat T$ exists. Then, 
  if $\hat T$ has only trivial fixed points,
  the ergodic equation~\eqref{eq:Ergodic} is solvable.
\end{theorem}

\section{Realizable mean payoffs}
\label{realizable-section}

We now show that the recession function of the Shapley operator
$T$ of the game $\Gamma(r,P)$ can be used to characterize the realizable mean payoff vectors of the games $\Gamma(r+g,P)$, where $g$ is a bounded
additive perturbation of the transition payment $r$.

Observe first that such a bounded additive perturbation $g$ of the transition payment $r$ does not change the recession function $\widehat{T}=T(0,P)$.
Moreover, combining Theorem~\ref{theo-GG} and 
Proposition~\ref{ergodic-solvable}, we get that if $\widehat{T}$
has only trivial fixed points, then $\chi(T)$ exists and is
a constant vector.
When the  mean payoff vector is already known to
exist, the following result, noted by several authors,
extends this assertion, since
it concerns also the case where $\hat T$ has non trivial fixed points.
\begin{proposition}[See~\cite{RS01,Sor04,GG04}]\label{impliesfixedpoint}
  Consider a game $\Gamma(r,P)$, such that  the recession function
  $\hat T$ and the mean payoff vector $\chi=\chi(T)$ 
  exist. Then $\hat{T}( \chi ) = \chi$.
\end{proposition}
We give the short proof for the convenience of the reader.
\begin{proof}
  Since $T$ is nonexpansive in the sup-norm $\| \cdot \|$, we have, for every vectors $x,y$ and every integer $n$,
  \[ \Big\| \frac{T(n x)-T(n y)}{n} \Big\| \leq \| x-y \| \enspace . \]
  Hence, taking $x=\chi$ and $y=T^n(0)/n$, we get
  \[ \Big\| \frac{T(n \chi)}{n} - \frac{T^{n+1}(\zero)}{n} \Big\| \leq \Big\| \chi - \frac{T^n(\zero)}{n} \Big\| \enspace . \]
  All the terms in the above inequality converge.
  Taking their limit, we obtain
  \[ \| \hat{T}(\chi) - \chi \| \leq \zero \enspace . \]
\end{proof}

We can also show a  converse statement, leading to the following equivalences.

\begin{proposition}[Realizable mean payoffs]
  \label{prop:GivenMeanPayoff}
  Let us fix a state space $S=\{1,\ldots,n\}$, and the actions
  spaces $A_i$ and $B_i$ of the two players.
  Consider a payment-free Shapley operator $F=T(0,P)$ 
  with transition probability $P$, and $T$ as in~\eqref{defshapley}.
  Then, the following assertions are equivalent:
  \begin{enumerate}
    \item\label{realmean1} 	$\nu \in \R^n$ is a fixed point of $F$;
    \item\label{realmean2} there exist a bounded transition payment $r$ such that the mean payoff vector of the game $\Gamma(r,P)$ exists and is equal to $\nu$;
    \item\label{realmean3} there exist a transition payment $r$ such that the recession function $\hat{T}$ and the mean payoff vector of the game $\Gamma(r,P)$ exist and are equal to $F$ and $\nu$ respectively.
  \end{enumerate}
\end{proposition}

\begin{proof}
  The implication~\pref{realmean2}$\Rightarrow$\pref{realmean3} is easy,
  and the implication~\pref{realmean3}$\Rightarrow$\pref{realmean1} 
  comes from Proposition~\ref{impliesfixedpoint}. Let us 
  show~\pref{realmean1}$\Rightarrow$\pref{realmean2}.
  
  Consider the transition payment $r$ such that $r_i^{a b} = \nu_i$ for every $i \in S$ and every $(a,b) \in A_i \times B_i$.
  The Shapley operator $T$ of the game $\Gamma(r,P)$ satisfies, by construction, $T(x) = F(x)+\nu$ for all $x\in\R^n$.
  
  For every integer $k$ we have $T(k \nu) = k F(\nu) + \nu = (k+1) \nu$, so that, by induction, $T^k(\zero) = k\nu$.
  This proves that the mean payoff vector of $\Gamma(r,P)$ 
  exists and is equal to $\nu$.
\end{proof}

Hence, for parametric games $\Gamma(\cdot,P)$ with fixed state space, action spaces and transition probability, the fixed points of the corresponding payment-free Shapley operator give exactly all the realizable mean payoff vectors.

We shall say that the game $\Gamma(r,P)$ is {\em ergodic} if it satisfies the conditions of the following theorem. 
\begin{theorem}[Ergodicity of zero-sum games]\label{th-game-ergodicity-full}
  Let us fix a state space $S=\{1,\ldots,n\}$, and the actions
  spaces $A_i$ and $B_i$ of the two players.
  Let $r$ be a bounded transition payment, $P$ be a transition probability,
  and let  $T=T(r,P)$ be the Shapley operator of the game $\Gamma(r,P)$.
  Then, the following properties are equivalent:
  \begin{enumerate}
    \item\label{game-ergodic-10} the recession function $\widehat{T}=T(0,P)$ has only trivial fixed points;
    \item\label{game-ergodic-20} the mean payoff vector of the game $\Gamma(r+g,P)$ does exist and is constant for all additive perturbations $g$ of the transition payment depending only of the state (so $g_i^{a,b}=g_i$, for all $i\in S$, $a\in A_i$ and $b\in B_i$);
    \item\label{game-ergodic-30} the ergodic equation $g + T(u)=\lambda \unit + u$ is solvable 
      for all vectors $g\in \R^n$;
    \item\label{game-ergodic-2p} the mean payoff vector of the game $\Gamma(r+g,P)$ does exist and is constant for all bounded additive perturbations $g$ of the transition payment $r$;
    \item\label{game-ergodic-3p} the ergodic equation $T'(u)=\lambda \unit + u$ 
      admits a solution $(\lambda,u)\in\R \times \R^n$, 
      for all Shapley operators $T'=T(r+g,P)$ associated to
      bounded additive perturbations $g$ of the transition payment.
  \end{enumerate}
\end{theorem}
\begin{proof}
  As said above, a bounded additive perturbation $g$ of the transition payment $r$
  does not change the recession function: $\widehat{T}=T(0,P)=\widehat{T'}$,
when $T'=T(r+g,P)$.
  Hence the implication
  \pref{game-ergodic-10}$\Rightarrow$\pref{game-ergodic-3p} 
  follows from Theorem~\ref{theo-GG}.
  Moreover, the implication
  \pref{game-ergodic-3p}$\Rightarrow$\pref{game-ergodic-2p} follows from 
  Proposition~\ref{ergodic-solvable}.
  Similarly, we have \pref{game-ergodic-30}$\Rightarrow$\pref{game-ergodic-20},
  since if $g$ is an additive perturbation of the transition payment
  depending only of the state (that is $g_i^{a,b}=g_i$, for all $i\in S$, 
  $a\in A_i$ and $b\in B_i$), then $T(r+g,P)=g+T(r,P)$.
  The implications \pref{game-ergodic-3p}$\Rightarrow$\pref{game-ergodic-30} 
  and \pref{game-ergodic-2p}$\Rightarrow$\pref{game-ergodic-20} are trivial.
  
  Hence, all the equivalences will follow from the implication 
  \pref{game-ergodic-20}$\Rightarrow$\pref{game-ergodic-10}, that we now prove.
  Assume that \pref{game-ergodic-20} holds.
  This means that the mean payoff vector of the game $\Gamma(r+g,P)$ does exist and is constant for all additive perturbations $g$ of the transition payment depending only of the state.
  Let $\eta$ be a fixed point of the recession function $\widehat{T}=T(0,P)$.
  Denote $T=T(r,P)$, and let $C>0$ be a bound of the transition payment $r$.
  We have 
  \begin{equation}\label{boundhat}
    -C+\widehat{T}(x)\leq T(x)\leq C+ \widehat{T}(x),\quad \forall x\in \R^n \enspace. 
  \end{equation}
  Let $s$ be an integer, consider the additive perturbation $g_s=s \eta$ of the transition
  payment, and denote $T_s=T(r+g_s,P)=g_s+T$. 
  Let us show by induction:
  \begin{equation}\label{boundts}
    k(s\eta -C)\leq (T_s)^k(0)\leq k(s\eta +C)\enspace.
  \end{equation}
  Indeed,  $T_s(0)=s\eta +T(0)$ and by~\eqref{boundhat}, we get
  that $-C\leq T(0)\leq C$, which shows~\eqref{boundts} for $k=1$.
  Assume that~\eqref{boundts} holds for $k\geq 1$.
  Then, by the monotonicity of $T_s$, we get that
  $(T_s)^{k+1}(0)\leq T_s(k(s\eta +C))$. Then, using the definition 
  of $T_s$, the additive homogeneity of $T$ and~\eqref{boundhat},
  we deduce:
  $T_s(k(s\eta +C))= s\eta +T(ks \eta+ kC)=s\eta +k C +T(k s\eta)
  \leq s\eta +(k+1) C +\widehat{T}(k s\eta)$.
  Since $\widehat{T}$ is positively homogeneous and $\eta$ is a fixed point
  of $\widehat{T}$\, we obtain that $\widehat{T}(ks\eta)=k s\eta$, hence
  $T_s(k(s\eta +C))\leq s\eta +(k+1) C +k s\eta=(k+1) (s\eta +C)$,
  which shows the second inequality of~\eqref{boundts} for $k+1$.
  The first inequality is obtained with the same arguments.
  
  Now, by~\pref{game-ergodic-20} the mean payoff vector 
  $\lim_{k\to\infty} (T_s)^k(0)/k=\chi_s$
  of the game $\Gamma(r+g_s,P)$ exists and is constant.
  From~\eqref{boundts}, we deduce that  
  \[s\eta -C \leq \chi_s \leq s\eta +C\enspace.\]
  Since $\chi_s$ is a constant vector, we get that 
  $s(\max_{i\in S} \eta_i) -C\leq (\chi_s)_{j}\leq  s(\min_{i\in S} \eta_i) +C$
  for all $j\in S$. 
  Hence, $s(\max_{i\in S} \eta_i -\min_{i\in S} \eta_i) \leq 2C$, and since this
  inequality holds for all $s>0$, we deduce that 
  $\max_{i\in S} \eta_i -\min_{i\in S} \eta_i=0$.
  This implies that $\eta$ is a constant vector,
  hence any fixed point of the recession
  function $\widehat{T}$ is a constant vector, which shows
  Assertion~\pref{game-ergodic-10}.
  
  Note that we could have shown the direct implication 
  \pref{game-ergodic-2p}$\Rightarrow$\pref{game-ergodic-10}, by using 
  the implication \pref{realmean1}$\Rightarrow$\pref{realmean2} in Proposition~\ref{prop:GivenMeanPayoff}.
\end{proof}

\section{Characterization of ergodicity in terms of Galois connections}
\label{galois-section}

In this section, we shall fix a state space $S=[n]$,
and consider any payment-free Shapley operator $F$ defined over $S$,
without specifying the actions spaces $A_i$ and $B_i$ of the two players 
nor the transition probability $P$.
Indeed, the results of this section only use the fact that
$F:\R^n\to \R^n$ is order-preserving, additively homogeneous,
and positively homogeneous.

\subsection{Invariant faces of the hypercube}

We begin with an observation about the fixed points of 
a payment-free Shapley operator.
But first, let us fix some notation.
If $K$ is a subset of $S$, denote by $\indic_K$ the vector with entries $1$ on $K$ and $0$ on $\Cpt{K}$.

\begin{lemma}
  \label{lem:NontrivialFP}
  Let $F$ be a payment-free Shapley operator.
  If $u$ is a nontrivial fixed point of $F$ then, denoting by $I = \argmin u$ and $J = \argmax u$, we have
  \begin{align}
    F(\indic_\Cpt{I}) &\leq \indic_\Cpt{I} \enspace , \tag{H1} \label{eq:H1} \\
    \indic_J &\leq F(\indic_J) \enspace . \tag{H2} \label{eq:H2}
  \end{align}
\end{lemma}

\begin{proof}
  By the additive and positive homogeneity of $F$, we may assume
  that $\indic_\Cpt{I} \leq u$ and that $\min_{s \in S} u_s = 0$.
  Hence, by the monotonicity of $F$, we get that $F(\indic_\Cpt{I}) \leq u$.
  In particular, we have $[F(\indic_\Cpt{I})]_i \leq 0$ for every $i \in I$.
  Since $\indic_\Cpt{I} \leq \unit$, we also have $F(\indic_\Cpt{I}) \leq \unit$ (recall that any trivial vector is a fixed point of $F$).
  It follows that $F(\indic_\Cpt{I}) \leq \indic_\Cpt{I}$.
  
  We show the second inequality using the same arguments (but this time assuming
  that $u \leq \indic_J$ and $\max_{s \in S} u_s = 1$).
\end{proof}

\begin{remark}
  \label{rem:H1H2}
  Note that the hypercube $[0,1]^n$ is invariant by $F$.
  Hence,~\eqref{eq:H1} is equivalent to the fact that $[F(\indic_\Cpt{I})]_i = 0$ for every $i \in I$.
  Likewise,~\eqref{eq:H2} is equivalent to $[F(\indic_J)]_j = 1$ for every $j \in J$.
\end{remark}

\begin{remark}
  Conditions~\eqref{eq:H1} and~\eqref{eq:H2} are dual.
  Indeed, introduce $\widetilde F$ the conjugate operator of $F$ defined by $\widetilde{F}(x):=-F(-x)$. 
  Then, $\widetilde F$ is a payment-free Shapley operator (obtained from $F$ by changing $\min$ to $\max$ and vice versa).
  Moreover, $\widetilde{F}(x)=\unit-F(\unit-x)$, so $\widetilde{F}(\indic_I)=\unit-F(\indic_\Cpt{I})$ for all subsets $I$ of $S$, and condition~\eqref{eq:H1} holds for $F$ and $I$ if, and only if, condition~\eqref{eq:H2} holds for $\widetilde F$ and $I$.
  
  Furthermore, if $u$ is a nontrivial fixed point of $F$, then the vector $\tilde u := \unit - u$ is a nontrivial fixed point of $\widetilde F$, verifying $\argmax \tilde u = \argmin u$.
\end{remark}

Conditions~\eqref{eq:H1} and~\eqref{eq:H2} can be stated in geometric terms.
Given two subsets $I$ and $J$ of $S$, denote by $C_I^- := \{x \in [0,1]^n \mid \forall i \in I,~x_i = 0\}$ and by $C_J^+ := \{x \in [0,1]^n \mid \forall j \in J,~x_j = 1\}$ two faces of the hypercube.
We shall call them lower and upper faces, respectively.
Note that they can alternatively be defined by $C_I^- = \{x \in [0,1]^n \mid x \leq \indic_\Cpt{I} \}$ and $C_J^+ = \{x \in [0,1]^n \mid x \geq \indic_J \}$.
Hence, by the monotonicity of $F$, we easily get the following.
\begin{proposition}
  \label{prop:Invariance}
  Let $F$ be a payment-free Shapley operator.
  Let $I$ and $J$ be two subsets of $S$. Then
  \begin{align*}
    \eqref{eq:H1} \enspace \Leftrightarrow \enspace F(C_I^-) \subset C_I^-,\\
    \eqref{eq:H2} \enspace \Leftrightarrow \enspace F(C_J^+) \subset C_J^+.
  \end{align*}
\end{proposition}

Conditions~\eqref{eq:H1} and~\eqref{eq:H2} are thus equivalent to the invariance of faces of the hypercube.

\subsection{Galois connection}
\label{section-galois}

We first recall the definition of a Galois connection between lattices, as introduced by Birkhoff~\cite{Bir95} for lattices of subsets and then generalized by Ore~\cite{Ore44}.
Let $(A,\prec_A)$ and $(B,\prec_B)$ be two partially ordered sets and let $\varphi : A \rightarrow B$ and $\gamma : B \rightarrow A$.
The map $\varphi$ is said to be antitone if $a \prec_A a'$ implies $\varphi(a')\prec_B \varphi(a)$.
The pair $(\varphi,\gamma)$ is a Galois connection between $A$ and $B$ if one of the following equivalent assertions is verified:
\begin{subequations}
  \begin{align}
    & \operatorname{id}_A \prec_A \gamma \circ \varphi, \quad \operatorname{id}_B \prec_B \varphi \circ \gamma, \quad \text{$\varphi$ and $\gamma$ are antitone}, \label{eq:Gal1} \\
    & \forall a \in A,~\forall b \in B, \quad \big(~b \prec_B \varphi(a) \Leftrightarrow a \prec_A \gamma(b)~\big) \enspace ,  \label{eq:Gal2} \\
    & \forall b \in B, \quad \gamma(b) = \text{max}_A \{ a;~b \prec_B \varphi(a) \} \enspace , \label{eq:Gal3} \\
    & \forall a \in A, \quad \varphi(a) = \text{max}_B \{ b;~a \prec_A \gamma(b)\} \enspace , \label{eq:Gal4}
  \end{align}
\end{subequations}
where, given a partially ordered set $(E,\prec_E)$, $\operatorname{id}_E$ is the identity map over $E$ and $\text{max}_E X$ states for the maximum of the subset $X \subset E$ with respect to the partial order $\prec_E$.

If $(\varphi, \gamma)$ is a Galois connection between $A$ and $B$, then $(\gamma, \varphi)$ is a Galois connection between $B$ and $A$, and according to~\eqref{eq:Gal3} (resp.\ to~\eqref{eq:Gal4}), $\gamma$ (resp.\ $\varphi$) is uniquely determined by $\varphi$ (resp.\ $\gamma$).
Denote by $\varphi^\star := \gamma$ and likewise by $\gamma^\star := \varphi$.
These maps have the following properties:
\begin{gather*}
  b = \varphi(a) \quad \Rightarrow \quad \varphi^\star(b) = \text{max}_A \{ a;~b = \varphi(a) \} \enspace , \\
  \varphi \circ \varphi^\star \circ \varphi = \varphi \quad \text{and} \quad \varphi^\star \circ \varphi \circ \varphi^\star = \varphi^\star \enspace , \\
  \big( \exists a \in A, \quad b = \varphi (a) \big) \quad \Leftrightarrow \quad \varphi \circ \varphi^\star (b) = b \enspace , \\
  ( \varphi^\star )^\star = \varphi \enspace.
\end{gather*}
We say that an element $a \in A$ (resp.\ $b \in B$) is closed with respect to the Galois connection $(\varphi,\varphi^\star)$ (resp.\ $(\varphi^\star,\varphi)$)
if $a = \varphi^\star \circ \varphi(a)$ (resp.\ $b = \varphi \circ \varphi^\star(b)$).
We can show that the set of closed elements in $A$ with respect to $(\varphi,\varphi^\star)$ is $\bar{A} := \varphi^\star(B)$ and that the set of closed elements in $B$ with respect to $(\varphi^\star,\varphi)$ is $\bar{B} := \varphi(A)$.
Then, $\varphi$ is an isomorphism from $\bar{A}$ to $\bar{B}$, and its inverse is $\varphi^\star$

Given a payment-free Shapley operator $F$, we denote by $\F^-$ (resp.\ $\F^+$) the families of subsets of $S$ verifying~\eqref{eq:H1} (resp.~\eqref{eq:H2}):
\begin{align*}
  \F^- &:= \big\{ I \subset S \mid F(\indic_\Cpt{I}) \leq \indic_\Cpt{I} \big\} \enspace ,\\
  \F^+ &:= \big\{ J \subset S \mid \indic_J \leq F(\indic_J) \big\} \enspace . 
\end{align*}
These families $\F^-$ and $\F^+$ are lattices of subsets with respect to the inclusion partial order. Indeed, since
$F$ is order preserving, for all $I_1,I_2\in \F^-$, we have
\begin{align*}
  F(\indic_{S\setminus (I_1\cup I_2)})&=F\big(\inf(\indic_{S\setminus I_1},\indic_{S\setminus I_2})\big)\\
  &\leq
  \inf\big(F(\indic_{S\setminus I_1}),F(\indic_{S\setminus I_2})\big) \\
  &\leq
  \inf(\indic_{S\setminus I_1},\indic_{S\setminus I_2}) 
  = \indic_{S\setminus (I_1\cup I_2)},
\end{align*}
so that $I_1\cup I_2\in \F^-$. This implies that the supremum
of two sets in $\F^-$ coincides with their supremum in the
powerset lattice $\mathcal{P}(S)$ of $S$, i.e., the union
$I_1\cup I_2$. Hence, $\F^-$ is
a sub-supsemilattice of $\mathcal{P}(S)$. Then,
since $\F^-$ has a bottom element (the empty set) 
and since it is a finite ordered set, it is
automatically an inf-semilattice: the infimum of two sets $I_1,I_2\in\F^-$
is given by $\cup_{I_3\in \F^-,I_3\subset I_1,I_3\subset I_2} I_3$.
Note that the latter infimum may differ from the infimum in $\mathcal{P}(S)$
(the intersection). The lattice $\F^+$ has dual properties.
According to the geometric interpretation, the two lattices $\F^-$ and $\F^+$ can be identified with the families of lower and upper invariant faces of the hypercube, respectively.
Note that $\F^-$ and $\F^+$ both contain
$\emptyset$ and $S$.

Given $I \in \F^-$, we are interested in the subsets $J \in \F^+$ satisfying $I \cap J = \emptyset$ (see Lemma~\ref{lem:NontrivialFP}). We shall
consider in particular
the greatest subset $J$ with the latter property.
Vice versa, starting from a subset $J$, we may consider the greatest subset 
$I$ with the same property.
In geometric terms, to each lower invariant face $C_I^-$ of $[0,1]^n$ we associate the smallest upper invariant face $C_J^+$ with nonempty intersection with
$C_I^-$.
This defines a Galois connection between the lattices $\F^-$ and $\F^+$.

Let $(\Phi,\Phi^\star)$ be the pair of functions from $\F^-$ (resp.\ $\F^+$) to $\F^+$ (resp.\ $\F^-$), that have just been introduced.
Formally, they are defined for every $I \in \F^-$ and $J \in \F^+$ by:
\begin{equation}
  \label{eq:GaloisConnection}
  \Phi(I) := \bigcup_{J \in \F^+, \; I \cap J = \emptyset} J \qquad \text{and} \qquad \Phi^\star(J) := \bigcup_{I \in \F^-, \; I \cap J = \emptyset} I \enspace .
\end{equation}
It follows from this definition that $\Phi$ and $\Phi^\star$ are antitone, and that $I \subset \Phi^\star \circ \Phi(I)$ and $J \subset \Phi \circ \Phi^\star(J)$.
Hence condition~\eqref{eq:Gal1} is  satisfied for the pair $(\Phi,\Phi^\star)$ which proves that it is a Galois connection between the lattices of subsets $\F^-$ and $\F^+$.

We now explore some properties of this Galois connection.
By a simple application of the definitions, we can first complete Lemma~\ref{lem:NontrivialFP}.
\begin{lemma}
  \label{lem:NontrivialFPandGaloisConnection}
  Let $F$ be a payment-free Shapley operator.
  If $u$ is a nontrivial fixed point of $F$, then $\argmin u \in \F^-$ and $\argmax u \in \F^+$.
  Furthermore, we have $\argmax u \subset \Phi(\argmin u)$ and $\argmin u \subset \Phi^\star(\argmax u)$.
  \qed
\end{lemma}

For $x \in \R^n$, we shall use the notation
\[ F^\omega(x) := \lim_{k \to \infty} F^k(x) 
\]
as soon as the latter limit exists. This is the case in particular when $F(x)\leq x$ or $x\leq F(x)$.
Indeed, since $F$ is order-preserving, the former (resp.\ latter) 
inequality implies
that the sequence $(F^k(x))_{k\geq 0}$ is nonincreasing (resp.\ nondecreasing). Moreover,
since $F$ is nonexpansive and has a fixed point (namely, $0$),
the sequence $(F^k(x))_{k\geq 0}$ is bounded, so that it converges as $k$ tends to infinity.
Moreover, $F^\omega(x)$ is necessarily a fixed point of $F$. 

\begin{lemma}
  \label{lem:ArgmaxPhi}
  Let $F$ be a payment-free Shapley operator.
  Let $I \in \F^-$ (resp.\ $J \in \F^+$) such that $\Phi(I) \neq \emptyset$ (resp.\ $\Phi^\star(J) \neq \emptyset$).
  Then, $\argmax F^\omega(\indic_\Cpt{I}) = \Phi(I)$ (resp.\ $\argmin F^\omega(\indic_J) = \Phi^\star(J)$).
  
  Furthermore, if $I$ (resp.\ $J$) is closed with respect to the Galois connection $(\Phi,\Phi^\star)$ (resp.\ $(\Phi^\star,\Phi)$), then $\argmin F^\omega(\indic_\Cpt{I}) = I$ (resp.\ $\argmax F^\omega(\indic_J) = J$).
\end{lemma}

\begin{proof}
  Firstly, note that since $F(\indic_\Cpt{I}) \leq \indic_\Cpt{I}$, 
  the sequence $(F^k(\indic_\Cpt{I}))_{k \geq 0}$
  is nonincreasing and so 
  the limit $u:=F^\omega(\indic_\Cpt{I})$ does exist.
  Since $F$ leaves $[0,1]^n$ invariant, we have $u\in [0,1]^n$.
  
  Secondly, by definition of the Galois connection, we have $\indic_{\Phi(I)} \leq \indic_\Cpt{I}$.
  Using the monotonicity of $F$ and the characterization of $\F^-$ and $\F^+$, we get that $\indic_{\Phi(I)} \leq F(\indic_{\Phi(I)}) \leq F(\indic_\Cpt{I}) \leq \indic_\Cpt{I}$.
  Since $F$ is monotone, we get that $\indic_{\Phi(I)} \leq 
  F^k(\indic_{\Phi(I)}) \leq F^k(\indic_\Cpt{I}) \leq \indic_\Cpt{I}$.
  It follows that $\indic_{\Phi(I)} \leq u \leq \indic_\Cpt{I}$ and we deduce that $I \subset \argmin u$ and $\Phi(I) \subset \argmax u$.
  Since $\Phi$ is antitone, we have $\Phi(\argmin u) \subset \Phi(I) \subset \argmax u$.
  
  The vector $u$ being a fixed point of $F$, we know from Lemma~\ref{lem:NontrivialFPandGaloisConnection} that $\argmin u \in \F^-$, $\argmax u \in \F^+$, and $\argmax u \subset \Phi(\argmin u)$.
  Hence, we have by the previous inclusions, $\Phi(\argmin u) = \Phi(I) = \argmax u$.
  
  Suppose now that $I$ is closed with respect to the Galois connection.
  This means that $\Phi^\star(\Phi(I)) = I$.
  Then, from the previous equalities, we get that $\Phi^\star \circ \Phi(\argmin u) = I$.
  This implies that $\argmin u \subset I$ and since we already know that $I \subset \argmin u$, we can conclude that $I = \argmin u$.
  
  The analogous results for $J \in \F^+$ follow by duality.
\end{proof}

We say that a subset of states is {\em proper} if it differs from the empty set and from the whole set of states.
We say that $I,J \subset S$ are {\em conjugate subsets of states} with respect to the Galois connection $(\Phi,\Phi^\star)$ if $I \in \F^- \setminus \emptyset$, $J \in \F^+ \setminus \emptyset$ and if $J = \Phi(I)$ and $I = \Phi^\star(J)$.  
\begin{theorem}
  \label{thm:Galois}
  Consider a payment-free Shapley operator $F$.
  The following assertions are equivalent:
  \begin{enumerate}
    \item $F$ has a nontrivial fixed point;\label{i-thm-Galois}
    \item there exist non empty disjoint subsets $I,J\subset S$ such that $I \in \F^-$ and $J \in \F^+$;\label{ii-thm-Galois}
    \item there exists a proper subset of states $I \in \F^-$ such that $\Phi(I)\neq \emptyset$;\label{iii-thm-Galois}
    \item there exists a proper subset of states $J \in \F^+$ such that $\Phi^\star(J)\neq \emptyset$;\label{iv-thm-Galois}
    \item there exists a proper subset of states that is closed with respect to the Galois connection $(\Phi,\Phi^\star)$ or $(\Phi^\star,\Phi)$;\label{v-thm-Galois}
    \item there exists a pair of conjugate subsets of states with respect to the Galois connection $(\Phi,\Phi^\star)$.\label{vi-thm-Galois}
  \end{enumerate}
\end{theorem}

\begin{proof}
  \pref{i-thm-Galois}$\Rightarrow$\pref{ii-thm-Galois}: Assume that
  \pref{i-thm-Galois} holds and consider a nontrivial fixed point $u$ of $F$.
  Then denoting by $I:=\argmin u$ and by $J:=\argmax u$, we know, by Lemma~\ref{lem:NontrivialFP}, that $I \in \F^-$ and $J \in \F^+$.
  Since $u$ is nontrivial, we also have $I \cap J = \emptyset$,
  which shows \pref{ii-thm-Galois}.
  
  \noindent 
  \pref{ii-thm-Galois}$\Rightarrow$\pref{iii-thm-Galois} and \pref{ii-thm-Galois}$\Rightarrow$\pref{iv-thm-Galois}:
  If \pref{ii-thm-Galois} holds, then by definition of the Galois connection we have $J \subset \Phi(I)$ and $I \subset \Phi^\star(J)$,
  thus  $\Phi(I)\neq \emptyset$ and $\Phi^\star(J)\neq \emptyset$, which shows 
  both \pref{iii-thm-Galois} and \pref{iv-thm-Galois}.
  
  \noindent
  \pref{iii-thm-Galois}$\Rightarrow$\pref{v-thm-Galois}:
  Let $I$ be a subset as in \pref{iii-thm-Galois},
  that is $I \in \F^-$ is proper such that $\Phi(I)\neq \emptyset$.
  We cannot have $\Phi(I) = S$, since otherwise, this would implies that 
  $I \subset \Phi^\star(\Phi(I))=\emptyset$. Hence, $\Phi(I)$ 
  is proper. Moreover, we know that $\Phi(I)$ is closed with respect to the Galois connection $(\Phi^\star,\Phi)$, which shows (one case of) \pref{v-thm-Galois}.
  
  \noindent
  \pref{iv-thm-Galois}$\Rightarrow$\pref{v-thm-Galois}:
  Similarly if $J \in \F^+$ is proper such that $\Phi^\star(J)\neq \emptyset$,
  then $\Phi^\star(J)$ is proper and closed with respect to the Galois 
  connection $(\Phi,\Phi^\star)$, which shows \pref{v-thm-Galois}.
  
  \noindent 
  \pref{v-thm-Galois}$\Rightarrow$\pref{i-thm-Galois}:	
  Suppose for instance that $I \in \F^-$ is proper and closed with respect to the Galois connection $(\Phi,\Phi^\star)$.
  We have $\Phi(I) \neq \emptyset$, since otherwise $I = \Phi^\star(\Phi(I))=S$.
  By the first assertion of Lemma~\ref{lem:ArgmaxPhi},
  we get that $u=F^\omega(\indic_{\Cpt{I}})$ is a fixed point of $F$ with an
  argmax equal to $\Phi(I)$. By the second assertion of
  Lemma~\ref{lem:ArgmaxPhi}, we obtain that $\argmin u=I$.
  Since $I \cap \Phi(I) = \emptyset$ and neither $I$ nor $\Phi(I)$
  is empty, $u$ is a nontrivial fixed point of $F$, which shows
  \pref{i-thm-Galois}.
  The same conclusion holds if there is a proper subset of states $J \in \F^+$, closed with respect to the Galois connection $(\Phi^\star,\Phi)$.

  \noindent
  \pref{v-thm-Galois}$\Leftrightarrow$\pref{vi-thm-Galois}:
  This follows readily from the definition of conjugate subsets of states.
\end{proof}

\section{Ergodicity only depends on the support of the transition probability}
\label{compact-section}

Let us fix now a state space $S=[n]$.
Here, given the actions        
spaces $A_i$ and $B_i$ of the two players, and the transition probability $P$,
we consider the payment-free Shapley operator $F=T(0,P)$,
with $T$ as in~\eqref{defshapley}.

\subsection{Boolean abstractions}

We call {\em upper} and {\em lower Boolean abstractions} of 
the payment-free Shapley operator $F$, the operators respectively defined on $\{0,1\}^n$ by
\begin{align}
  [F^+(x)]_i & := \min_{a \in A_i} \max_{b \in B_i} \max_{j:(P_i^{a b})_j > 0} x_j, \quad i \in S \enspace , \label{defF+}\\
  [F^-(x)]_i & := \min_{a \in A_i} \max_{b \in B_i} \min_{j:(P_i^{a b})_j > 0} x_j, \quad i \in S \enspace . \label{defF-}
\end{align}
These Boolean operators can be extended to $\R^n$.
Then, we have $F^- \leq F \leq F^+$.

We now make some observations.
Firstly, the expressions of $F^+$ and $F^-$ involve the operators $\min$ and $\max$ (instead of $\inf$ and $\sup$).
This owes to the fact that the action spaces are nonempty, and the state space is finite, hence, given $x \in \R^n$, the $\min$ and $\max$ operations are applied to nonempty subsets of the finite set $\{x_i\}_{i \in S}$. 
Secondly, $F^+$ and $F^-$ are only determined by the support of the transition probability, that is the set of $(i,a,b,j)$ such that $(P_{i}^{a b})_j > 0$.
Finally, $\big(\widetilde{F}\big)^+ = \widetilde{(F^-)}$, recalling that $\widetilde F$ is the conjugate operator of $F$ defined by $\widetilde F(x)=\ -F(-x)$.

These Boolean operators are helpful to characterize the families $\F^-$ and $\F^+$ as well as the Galois connection $(\Phi, \Phi^\star$).
However,
we need to make the following assumption.
\begin{assumption}
  \label{StandardAssumptions}
  \begin{enumerate}
    \item[] 
    \item For every state $i \in S$, the action spaces $A_i$ and $B_i$ are nonempty compact sets;
    \item The transition probability $P$ is separately continuous, meaning that given $i \in S$ and $a \in A_i$ the function $b \in B_i \mapsto P_i^{a b} \in \Delta(S)$ is continuous, and given $i \in S$ and $b \in B_i$ the function $a \in A_i \mapsto P_i^{a b} \in \Delta(S)$ is also continuous.
  \end{enumerate}
\end{assumption}

This assumption implies in particular the existence of 
optimal policies for both players, a property which is 
used implicitly in the proof of the following result.

\begin{lemma}
  \label{lem:BooleanReformulation}
  Let $F$ be the payment-free Shapley operator
  associated with the actions spaces $A_i$ and $B_i$ of the two players, 
  and the transition probability $P$, and let $F^+$ and $F^-$
  be defined by~\eqref{defF+} and~\eqref{defF-} respectively.
  For all subsets $I$ and $J$ of $S$, consider the conditions:
  \begin{align}
    F^+(\indic_\Cpt{I}) &\leq \indic_\Cpt{I} \enspace , \tag{H1'} \label{eq:H1'} \\
    \indic_J &\leq F^-(\indic_J) \enspace . \tag{H2'} \label{eq:H2'}
  \end{align}
  We have \eqref{eq:H1'} $\Rightarrow$  \eqref{eq:H1} and
  \eqref{eq:H2'} $\Rightarrow$ \eqref{eq:H2}.
  Moreover, under Assumption~\ref{StandardAssumptions}, we have:
  \eqref{eq:H1} $\Rightarrow$  \eqref{eq:H1'} and
  \eqref{eq:H2} $\Rightarrow$ \eqref{eq:H2'}.
\end{lemma}

\begin{proof}
  Since $F^-\leq F\leq F^+$, the implications \eqref{eq:H1'} $\Rightarrow$ \eqref{eq:H1} and \eqref{eq:H2'} $\Rightarrow$ \eqref{eq:H2} are
  trivial. It remains to show the reverse of these implications
  under Assumption~\ref{StandardAssumptions}.
  
  Let us first make some observations.
  First, for all $x\in [0,1]^n$, $i\in S$, $a\in A_i$ and
  $b\in B_i$, we have
  \begin{equation}\label{equiv0} P_i^{a b} x \leq 0\Leftrightarrow
    (P_i^{a b})_j = 0 \; \text{or} \; x_j = 0\quad \forall j \in S
    \Leftrightarrow \max_{j:(P_{i}^{a b})_j > 0} x_j = 0\enspace,\end{equation}
  since all entries of $P_i^{a b}$ and $x$ are nonnegative.
  Similarly
  \begin{equation}\label{equiv1} P_i^{a b} x \geq 1\Leftrightarrow
    \min_{j:(P_{i}^{a b})_j > 0} x_j = 1\enspace,\end{equation}
  since all entries of $x$ are less or equal to $1$,
  and  $P_i^{a b} x= 1-P_i^{a b} (1-x)$.
  
  Next, for all $x\in\R^n$, by Assumption~\ref{StandardAssumptions}, for all $i \in S$ and $a \in A_i$
  there exists $b\in B_i$ depending on $i$ and $a$ such that
  $ P^{ab}_i x= \max_{b'\in B_i} P^{ab'}_i x$ (the supremum is thus a maximum).
  Assumption~\ref{StandardAssumptions} also implies
  that, for all $i\in S$ and $b\in B_i$, the map 
  $A_i \to \R, \; a\mapsto  P^{ab}_i x$ is continuous.
  Since the supremum of continuous maps is lower 
  semicontinuous, we get that for all $i\in S$, the map 
  $A_i \to \R, \; a \mapsto  \max_{b\in B_i} P^{ab}_i x$
  is  lower semicontinuous.
  Since $A_i$ is compact, this implies that, for all $i\in S$, there exists 
  $a\in A_i$ such that
  $\max_{b'\in B_i} P^{ab'}_i x= \min_{a'\in A_i} \max_{b'\in B_i} P^{a'b'}_ix=[F(x)]_i$.
  
  Let us now show \eqref{eq:H1}$\Rightarrow$\eqref{eq:H1'}.
  Assume \eqref{eq:H1}, that is $F(\indic_\Cpt{I}) \leq \indic_\Cpt{I}$.
  This implies that $[F(\indic_\Cpt{I})]_i \leq 0$ for all $i\in I$. 
  By the last observation above, given $i\in I$, 
  there exists $a\in A_i$ such that
  $\max_{b\in B_i}  (P^{ab}_i \indic_\Cpt{I})=[F(\indic_\Cpt{I})]_i\leq 0$.
  Then, for all $b\in B_i$, $P^{ab}_i \indic_\Cpt{I}\leq 0$,
  which implies, by~\eqref{equiv0}, 
  that $ \max_{j:(P_{i}^{a b})_j > 0} [\indic_\Cpt{I}]_i = 0$.
  Since this last equality holds for some $a\in A_i$ and all $b\in B_i$, we 
  deduce that
  \[ [F^+(\indic_\Cpt{I})]_i=\min_{a \in A_i} \max_{b \in B_i} \max_{j: (P_{i}^{a b})_j > 0} [\indic_\Cpt{I}]_j =0,\quad \text{for all}\;i \in I\enspace,\]
  hence $F^+(\indic_\Cpt{I}) \leq \indic_\Cpt{I}$.
  
  With similar arguments, we show that \eqref{eq:H2}$\Rightarrow$\eqref{eq:H2'}.
\end{proof}

In the sequel, we shall also consider the sets 
$\F'^-$ and $\F'^+$ defined like $\F^-$ and $\F^+$, but with 
the conditions~\eqref{eq:H1'} and \eqref{eq:H2'}
instead of the conditions~\eqref{eq:H1} and \eqref{eq:H2} respectively:
\begin{align}
  \label{defcalF'-}
  \F'^- &:= \big\{ I \subset S \mid F^+(\indic_\Cpt{I}) \leq \indic_\Cpt{I} \big\},\\
  \label{defcalF'+}
  \F'^+ &:= \big\{ J \subset S \mid \indic_J \leq F^-(\indic_J) \big\}.
\end{align}
Moreover, we shall denote by $\Phi'$ and $\Phi'^*$ the maps
defined by~\eqref{eq:GaloisConnection}
with $\F'^-$ and $\F'^+$ instead of $\F^-$ and $\F^+$  respectively.
Then, $(\Phi', \Phi'^\star)$ is also a Galois connection between the
lattices of subsets $\F'^-$ and $\F'^+$.

Remark that from the definitions, $\F'^-$ and $\F'^+$ and then $(\Phi', \Phi'^\star)$ only depend on the support of the transition probability P, i.e.\ the set of elements $(i,a,b,j)$ such that $i,j \in S$, $a \in A_i$, $b \in B_i$, and $(P_i^{a b})_j > 0$.
Furthermore, Lemma~\ref{lem:BooleanReformulation} shows that under
Assumption~\ref{StandardAssumptions}, the former new sets and maps 
coincide with the corresponding old ones.

\begin{corollary}
  \label{coro:StructuralGaloisConnection}
  Given the payment-free Shapley operator $F$ 
  associated with actions spaces $A_i$ and $B_i$ of the two players, 
  and a transition probability $P$, 
  the families $\F^+$ and $\F^-$, as well as the Galois connection $(\Phi, \Phi^\star)$, depend only on the support of the transition probability P, when Assumption~\ref{StandardAssumptions} holds.
\end{corollary}

Using Theorem~\ref{thm:Galois} together with the previous result, we deduce the following one.
\begin{corollary}
  \label{coro:StructuralFixedPoints}
  Given the payment-free Shapley operator $F$ 
  associated with actions spaces $A_i$ and $B_i$ of the two players, 
  and a transition probability $P$, 
  such that Assumption~\ref{StandardAssumptions} holds,
  the property ``$F$ has only trivial fixed points''
  depends only on the support of the transition probability P.
\end{corollary}

The theorem below gives a way to compute the images of subsets of states by the Galois connection.

\begin{theorem}
  \label{thm:GaloisComputation}
  Let $F$ be the payment-free Shapley operator
  associated with actions spaces $A_i$ and $B_i$ of the two players, 
  and a transition probability $P$, let $F^+$ and $F^-$
  be defined by~\eqref{defF+} and~\eqref{defF-} respectively.
  For $I \in \F'^-$ and $J \in \F'^+$ we have
  \begin{align*}
    \indic_{\Phi'(I)} &= (F^-)^\omega(\indic_\Cpt{I}) \enspace ,\\
    \indic_\Cpt{\Phi'^\star(J)} &= (F^+)^\omega(\indic_J) \enspace .
  \end{align*}
\end{theorem}

\begin{proof}
  We show only the first assertion, the second follows by duality.
  
  Let $I\in\F'^-$. By definition, $F^+(\indic_\Cpt{I}) \leq \indic_\Cpt{I}$,
  and using $F^-\leq F^+$, we obtain $F^-(\indic_\Cpt{I}) \leq \indic_\Cpt{I}$.
  It follows that $(F^-)^\omega(\indic_\Cpt{I})$ is well defined
  and $(F^-)^\omega(\indic_\Cpt{I})\leq  \indic_\Cpt{I}$.
  $F^-$ is a Boolean map, hence there exist $L\subset S$ such that 
  $\indic_L = (F^-)^\omega(\indic_\Cpt{I})$. It remains to show that $L=\Phi'(I)$.
  
  Since $\indic_L$ is a fixed point of $F^-$, $L$ belongs to $\F'^+$.
  Furthermore, it satisfies $\indic_L \leq \indic_\Cpt{I}$, that is, $I \cap L = \emptyset$. Then $L\subset \Phi'(I)$.
  
  Let $K \in \F'^+$ such that $I \cap K = \emptyset$, that is, $\indic_K \leq \indic_{S \setminus I}$.
  By induction, we get that $(F^-)^k(\indic_K) \leq (F^-)^k(\indic_\Cpt{I})$ for every integer $k$.
  By definition, we also 
  have that  $\indic_K \leq F^-(\indic_K)$, hence  $(F^-)^\omega(\indic_K)$
  exists and $(F^-)^\omega(\indic_K)\geq \indic_K$.
  This leads to $\indic_K \leq (F^-)^\omega(\indic_K)\leq 
  (F^-)^\omega(\indic_\Cpt{I})=\indic_L$, which implies that $ K \subset L$.
  This holds for all $K \in \F'^+$ such that $I \cap K = \emptyset$, 
  hence, by definition of $\Phi'$, $\Phi'(I) \subset L$.
\end{proof}

We conclude this subsection by giving an interpretation in terms of zero-sum game of the conditions \eqref{eq:H1'} and \eqref{eq:H2'} of 
Lemma~\ref{lem:BooleanReformulation}.

\begin{proposition}
  \label{prop:GameInterpretation}
  Let us fix a state space $S = [n]$, and the actions spaces $A_i$ and $B_i$ of the two players.
  Let $r$ be a bounded transition payment, $P$ be a transition probability, and let $T=T(r,P)$ be the Shapley operator of the game $\Gamma(r,P)$.
  Then
  \begin{enumerate}
    \item \label{minforce}
      \eqref{eq:H1'} holds for $F=\hat T$ and $I \subset S$ if, and only if, there exists a policy for player \MIN, i.e.\ a map $i\in S \mapsto a \in A_i$ such that for any strategy of player \MAX\ and any initial state in $I$, the sequence of states of the game $\Gamma(r,P)$ stays in $I$ almost surely;
    \item \label{maxforce}
      \eqref{eq:H2'} holds for $F=\hat T$ and $J \subset S$ if, and only if, there  exists a policy for player \MAX, i.e.\ a map $(i,a)\in \cup_{i \in S}(\{i\} \times A_i) \mapsto b \in B_i$ such that  for any strategy of player \MIN\ and any initial state in $J$,  the sequence of states of the game $\Gamma(r,P)$ stays in $J$ almost surely.
  \end{enumerate}
\end{proposition}

\begin{proof} 
  \pref{minforce}: Suppose that~\eqref{eq:H1'} holds for $F=\hat T$ and $I \subset S$. 
  Then, for all $i\in I$, we have $[F^+(\indic_\Cpt{I})]_i = 0$.
  Since $F^+$ involves min and max operators (see~\eqref{defF+}),
  we obtain that, for every $i \in I$, 
  there is an action $a \in A_i$ of player \MIN\ such that 
  $ \max_{j:(P_{i}^{a b})_j > 0} [\indic_\Cpt{I}]_j = 0$
  for every action $b \in B_i$ of player \MAX,
  which is equivalent, by~\eqref{equiv0}, with $P_i^{a b} \indic_\Cpt{I}=0$.
  Since $S$ is finite, there exists an element $\sigma$ of $\A$,
  that is a policy $\sigma$ of player \MIN, 
  $\sigma: i\in S\mapsto \sigma(i)\in A_i$,
  such that $P_i^{\sigma(i) b} \indic_\Cpt{I}=0$ for all $i\in I$ and  $b \in B_i$.
  
  Denote, as in Section~\ref{sec-games-basic},
  by $(i_k)_{k\geq 0}$ the (random) sequence 
  of states of the game  $\Gamma(r,P)$.
  If the current state $i_k$ is in $I$, then the probability that the state $i_{k+1}$ at the following stage is in $\Cpt{I}$ is equal to $P_i^{a b} \indic_\Cpt{I}$ if actions $a$ and $b$ are chosen.
  In particular, if player  \MIN\ selects the action $\sigma(i)$, then this probability is $0$, whatever player \MAX\ chooses.
  Hence, if player \MIN\  chooses the Markovian stationary strategy 
  corresponding to $\sigma$ ($a_k=\sigma(i_k)$ for all $k\geq 0$), and if the initial state $i_0$ is in $I$, then for any strategy (Markovian or not) of player \MAX, the probability that the sequence of states  $(i_k)_{k\geq 0}$ leaves $I$
  is $0$. This shows the ``only if'' part of \pref{minforce}.
  
  Conversely, suppose that there exists a policy $\sigma: i\in S \mapsto \sigma(i) \in A_i$ of player \MIN\ such that for any initial state $i_0$ in $I$, if player \MIN\  chooses the Markovian stationary strategy corresponding to $\sigma$, then (for any strategy of player \MAX), the state of the game $\Gamma(r,P)$ stays in $I$ almost surely.
  In particular, for any $i \in I$ and any $b \in B_i$,
  taking $i_0=i$, the strategy $a_k=\sigma(i_k)$, $k\geq 0$, for
  player \MIN\ and any strategy of player \MAX\ such that $b_0=b$, we get that 
  the probability that $i_1$ is outside $I$ is equal to $0$.
  Since this probability coincides with $P_i^{\sigma(i) b} \indic_\Cpt{I}$, 
  we deduce, using~\eqref{equiv0}, that
  $ \max_{j:(P_{i}^{\sigma(i) b})_j > 0} [\indic_\Cpt{I}]_j = 0$.
  This holds for all $b\in B_i$ and $i \in I$,
  hence $[F^+(\indic_\Cpt{I})]_i\leq \max_{b\in B_i} \max_{j:(P_{i}^{\sigma(i) b})_j > 0} [\indic_\Cpt{I}]_j = 0$ for all  $i \in I$.
  It follows that $F^+(\indic_\Cpt{I}) \leq \indic_\Cpt{I}$, 
  that is~\eqref{eq:H1'}.
  
  \pref{maxforce}: Suppose that~\eqref{eq:H2'}
  holds for $F=\hat T$ and $J \subset S$. 
  Then, for all $i\in J$, we have $[F^-(\indic_J)]_i = 1$.
  Since $F^-$ involves min and max operators (see~\eqref{defF-}),
  we obtain that, for every $i \in J$ and $a \in A_i$
  there is an action $b \in B_i$ of player \MAX\ such that 
  $ \min_{j:(P_{i}^{a b})_j > 0} [\indic_J]_j = 1$, 
  which is equivalent, by~\eqref{equiv1}, with $P_i^{a b} \indic_J=1$.
  By the axiom of choice, there exists a map 
  $\tau: (i,a)\in \cup_{i\in S} (\{i\}\times A_i)\mapsto \tau(i,a)\in B_i$,
  that is a policy of player \MAX, 
  such that $P_i^{a\tau(i,a)} \indic_J=1$ for all $i\in J$ and $a \in A_i$.
  
  By the same arguments as above, we get that for the game $\Gamma(r,P)$, 
  if player \MAX\ chooses the Markovian stationary strategy 
  corresponding to $\tau$ ($b_k=\tau(i_k,a_k)$ for all $k\geq 0$), and if the initial state $i_0$ is in $J$, then for any strategy (Markovian or not) 
  of player \MIN, the probability that the sequence of states $(i_k)_{k\geq 0}$ 
  leaves $J$ is $0$. This shows the ``only if'' part of \pref{maxforce}.
  The ``if'' part is obtained by the same arguments as for~\pref{minforce}.
\end{proof}

\subsection{Hypergraph characterization}
\label{sec-hypergraph}

In this subsection, we introduce directed hypergraphs which will allow us to represent the Boolean operators $F^+$ and $F^-$.
In particular, we shall see that 
finding $\Phi'(I)$ (resp.\ $\Phi'^\star(J)$) for a given $I \in \F'^-$ (resp.\ $J \in \F'^+$),
is equivalent to solving a reachability problem in a directed hypergraph.
We refer the reader to~\cite{GLNP93,AllamigeonAlgorithmica2013} for
more background on reachability problems in hypergraphs.

A \textit{directed hypergraph} is a pair $(N,E)$, where $N$ is a set of \textit{nodes} and $E$ is a set of (directed) \textit{hyperarcs}.
A hyperarc $e$ is an ordered pair $(\tail(e),\head(e))$ of disjoint nonempty subsets of nodes; $\tail(e)$ is the \textit{tail} of $e$ and $\head(e)$ is its \textit{head}. We shall often write $\tail$ and $\head$ instead of $\tail(e)$ and $\head(e)$, respectively, for brevity. When $\tail$ and $\head$ are both of cardinality one, the hyperarc is said to be an arc, and when every hyperarc is an arc,
the directed hypergraph becomes a directed graph.

In the following, the term {\em hypergraph}
will always refer to a directed hypergraph. 
The {\em size} of a hypergraph $G=(N,E)$ is defined as $\size(G) = |N|+\sum_{e \in E} |\tail(e)|+|\head(e)|$, where $|X|$ denotes the cardinality of any set $X$.
Note that we shall consider in the sequel hypergraphs with an infinite
number of nodes or hyperarcs, leading to $\size(G)=\infty$
(we set $|X|=\infty$ when $X$ is infinite).

Let $G=(N,E)$ be a hypergraph.
A {\em hyperpath} of length $p$ from a set of nodes $I \subset N$ to a node $j \in N$
is a sequence of $p$ hyperarcs $(\tail_1,\head_1),\dots,(\tail_p,\head_p)$,
such that $\tail_i \subset \cup_{k=0}^{i-1} \head_k$ for all $i=1, \dots, p+1$
with the convention $\head_0=I$ and $\tail_{p+1}=\{j\}$.
Then, we say that a node $j \in N$ is {\em reachable} from a set $I \subset N$,
if and only if there exists a hyperpath from $I$ to $j$.
Alternatively, the relation of reachability can be defined in a recursive way:
a node $j$ is reachable from the set $I$ if either $j \in I$
or there exists a hyperarc $(\tail,\head)$ such that $j \in \head$
and every node of $\tail$ is reachable from the set $I$.
A set $J$ is said to be {\em reachable} from a set $I$
if every node of $J$ is reachable from $I$. We denote by
$\Reach(I,G)$ the set of reachable nodes from $I$.

A subset $I$ of $N$ is {\em invariant} in the hypergraph $G$
if it contains every node that is
reachable from itself, that is $\Reach(I,G)\subset I$.
If $N'\subset N$, 
we shall also say that a subset $I$ of $N'$ is {\em invariant} 
in the hypergraph $G$ relatively to $N'$,
if it contains every node of $N'$ that is
reachable from itself, that is $\Reach(I,G)\cap N'\subset I$.
One readily checks that the
set of nodes of $N'$ that are reachable from a given set $I\subset N'$
is the smallest invariant set in the hypergraph $G$ relatively to $N'$,
containing $I$. The reachability notion will be illustrated in Example~\ref{ex-reach}.

We now make the connection with our problem.
Let $F$ be the payment-free Shapley operator
associated with the actions spaces $A_i$ and $B_i$ of the two players, 
and the transition probability $P$, and let $F^+$ and $F^-$
be its Boolean abstractions 
defined by~\eqref{defF+} and~\eqref{defF-} respectively.
We construct two hypergraphs $G^+ = (N^+,E^+)$ (Figure~\ref{fig:HyperG+}) and $G^- = (N^-,E^-)$ (Figure~\ref{fig:HyperG-}) as follows.
We first need to introduce a copy of $S$, denoted by $S'$.
It is a set disjoint from $S$ and given by a bijection $\pi : S \to S'$.
For the purposes of the following constructions, we also need to assume $S'$ disjoint from the two sets $\{ (i,a) \mid i \in S, \; a \in A_i \}$ and $ \{ (i,a,b) \mid i \in S, \; a \in A_i, \; b \in B_i\}$.

The node set of $G^+$ is
$N^+ = \{ (i,a) \mid i \in S, \; a \in A_i \} \cup S \cup S'$.
The hyperarcs of $G^+$ are of the form:
\begin{itemize}
  \item[-] $( \{i\} \times A_i, \{ \pi(i) \} ), \enspace i \in S$;
  \item[-] $( \{j\}, \{ (i,a) \} )$, for all $j,i \in S$ and $a \in A_i$ such that there exists $b \in B_i$ with $(P_{i}^{ab})_j> 0$.
\end{itemize}
As shown on Figure~\ref{fig:HyperG+}, this hypergraph is structured in two layers;
the first layer consists of the arcs $( \{j\}, \{(i,a)\} )$
whereas the second layer consists of the hyperarcs $( \{i\} \times A_i, \{ \pi(i) \} )$.

\tikzset{->-/.style={decoration={markings,mark= at position 0.5 with {\arrow{#1}},},postaction={decorate}}}
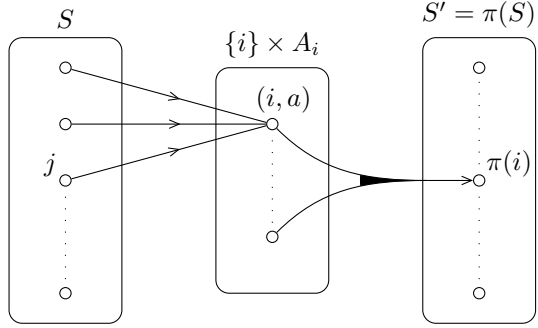
\begin{figure}[htbp]
  \begin{center}
    \begin{tikzpicture}
      [node distance=0.75cm, on grid, bend angle=20, auto,
        state/.style={circle,draw,inner sep=0pt,minimum size=1.5mm},
        dots/.style={-,shorten >=3pt,shorten <=3pt,loosely dotted},
        simpledge/.style={->-={angle 45 reversed}},
        hyperedge/.style={>=angle 45},
        frame/.style={draw,rectangle,rounded corners=5pt}]
      \node[state] (s_1) at (0,0) {};
      \node[state] (s_2) [below=of s_1] {};
      \node[state] (s_3) [below=of s_2, label={[shift={(-0.2,-0.15)}]$j$}] {};
      \node[state] (s_4) [below=of s_3, yshift=-0.75cm] {}
        edge [dots] (s_3);
      \node[state] (a_1) [right=of s_2, xshift=2.0cm,label={[shift={(0.15,-0.05)}]${(i,a)}$}] {}
        edge [simpledge] (s_1)
        edge [simpledge] (s_2)
        edge [simpledge] (s_3);
      \node[state] (a_2) [below=of a_1, yshift=-0.75cm] {}
        edge [dots] (a_1);
      \node[state] (s_5) [right=of s_3, xshift=4.75cm, label={[shift={(0.4,-0.15)}]${\pi(i)}$}] {};
      \node[state] (s_6) [above=of s_5, yshift=0.75cm] {}
        edge [dots] (s_5);
      \node[state] (s_7) [below=of s_5, yshift=-0.75cm] {}
        edge [dots] (s_5);
      \hyperedge[0.40][$(hyper@tail)!0.75!(hyper@head)$]{a_1,a_2}{s_5};
      \node[frame,minimum height=3.8cm,minimum width=1.5cm] [below=of s_2,label=above:$S$] {};
      \node[frame,minimum height=3.0cm,minimum width=1.5cm] [below=of a_1,label=above:$\{i\}\times A_i$] {};
      \node[frame,minimum height=3.8cm,minimum width=1.5cm] [right=of s_3, xshift=4.75cm,label=above:${S'=\pi(S)}$] {};
    \end{tikzpicture}
  \end{center}
  \caption{Hypergraph $G^+$ associated with $F^+$}
  \label{fig:HyperG+}
\end{figure}

The motivation for the construction of this hypergraph is that it encodes the Boolean operator $F^+$.
Recall that for $x \in \{0,1\}^n$, we have
\[ [F^+(x)]_i = \min_{a \in A_i} \max_{b \in B_i} \max_{j:(P_i^{a b})_j > 0} x_j \; , \enspace i \in S \enspace . \]
Denoting by $y_{i,a} := \max_{b \in B_i} \max_{j:(P_i^{a b})_j > 0} x_j$, we also have $[F^+(x)]_i = \min_{a \in A_i} y_{i,a}$.
If $x = \indic_J$ for some $J \subset S$, then $y_{i,a} = 1$ if, and only if, there is $b \in B_i$ and $j \in J$ such that $(P_i^{a b})_j > 0$.
This is also equivalent to the node $(i,a)$ being reachable from $J$ in $G^+$.
Then, $[F^+(x)]_i = 1$ if, and only if, $y_{i,a} = 1$ for every $a \in A_i$, which is equivalent to all the nodes in the tail of the hyperarc $(\{i\} \times A_i, \{\pi(i)\})$ being reachable from $J$ in $G^+$.
According to the recursive definition of reachability, this is equivalent to $\pi(i)$ being reachable from $J$ in $G^+$.
Hence, we have the following result.

\begin{proposition}
  \label{prop:HypergraphG+}
  Let $F$ be the payment-free Shapley operator associated with the actions spaces $A_i$ and $B_i$ of the two players, and the transition probability $P$, and let $F^+$ be defined by~\eqref{defF+}.
  Then, the node $\pi(i) \in S'$ is reachable from $J \subset S$ in $G^+$ if, and only if, $[F^+(\indic_J)]_i = 1$.
  \qed
\end{proposition}

\begin{example}\label{ex-reach}
  Let us consider the following payment-free Shapley operator defined on $\R^3$:
  \begin{equation*}
    F(x) =
    \begin{pmatrix}
      (x_1 \vee x_2) \, \wedge \, \frac{1}{2} (x_1 + x_3)\\
      x_1 \, \wedge \, \frac{1}{2} (x_2 + x_3)\\
      x_2 \, \vee \, x_3
    \end{pmatrix}
  \end{equation*}
  where $\wedge$ stands for $\min$ and $\vee$ for $\max$.
  The Boolean operator $F^+$ associated with $F$ is defined by
  \begin{equation*}
    F^+(x) =
    \begin{pmatrix}
     ( x_1 \, \vee  x_2) \, \wedge (x_1 \, \vee \, x_3)\\
      x_1 \, \wedge \, (x_2 \, \vee \, x_3)\\
      x_2 \vee x_3
    \end{pmatrix} \enspace .
  \end{equation*}
  Figure~\ref{fig:ExampleG+} shows the hypergraph $G^+$ associated with $F$,
where the element $\pi(i)$  of $S'$ is denoted $i'$.
  It can be checked that the nodes $1'$ and $3'$ are reachable from $\{2,3\}$, whereas the node $2'$ is not.
  According to Proposition~\ref{prop:HypergraphG+}, this is equivalent to the fact that
  \[
  F^+(\indic_{\{2,3\}}) = \indic_{\{1,3\}} \enspace .
  \]
  \tikzset{->-/.style={decoration={markings,mark= at position 0.42 with {\arrow{#1}},},postaction={decorate}}}
  \begin{figure}[htbp]
  \begin{center}
    \begin{tikzpicture}
      [node distance=1.15cm, on grid, bend angle=20, auto,
        state/.style={circle,draw,inner sep=0pt,minimum size=1.5mm},
        simpledge/.style={->-={angle 45 reversed}},
        hyperedge/.style={>=angle 45},
        frame/.style={draw,rectangle,rounded corners=5pt}]
      \node[state] (s_1) at (0,0) [label={[shift={(-0.2,-0.1)}]1}] {};
      \node[state] (s_2) [below=of s_1,label={[shift={(-0.2,-0.1)}]2}] {};
      \node[state] (s_3) [below=of s_2, label={[shift={(-0.2,-0.1)}]3}] {};
      \node[state] (a_2) [right=of s_1, xshift=1.25cm] {}
        edge [simpledge] (s_1)
        edge [simpledge] (s_3);
      \node[state] (a_1) [above=of a_2] {}
        edge [simpledge] (s_1)
        edge [simpledge] (s_2);
      \node[state] (a_3) [below=of a_2] {}
        edge [simpledge] (s_1);
      \node[state] (a_4) [below=of a_3] {}
        edge [simpledge] (s_2)
        edge [simpledge] (s_3);
      \node[state] (a_5) [below=of a_4] {}
        edge [simpledge] (s_2)
        edge [simpledge] (s_3);
      \node[state] (ss_1) [right=of a_2, xshift=1.25cm, label={[shift={(0.2,-0.1)}]1'}] {};
      \node[state] (ss_2) [below=of ss_1, label={[shift={(0.2,-0.1)}]2'}] {};
      \node[state] (ss_3) [below=of ss_2, label={[shift={(0.2,-0.1)}]3'}] {};
      \draw [->, >=angle 45] (a_5) -- (ss_3);
      \hyperedge[0.45][$(hyper@tail)!0.65!(hyper@head)$]{a_1,a_2}{ss_1};
      \hyperedge[0.50][$(hyper@tail)!0.65!(hyper@head)$]{a_3,a_4}{ss_2};
    \end{tikzpicture}
  \end{center}
  \caption{The hypergraph $G^+$ associated with $F$}
  \label{fig:ExampleG+}
\end{figure}
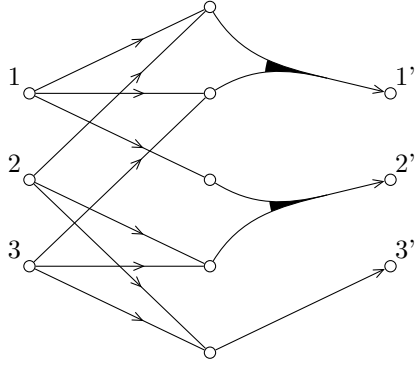
\end{example}

The node set of $G^-$ is $N^- = \{(i,a,b) \mid i \in S, \; a \in A_i, \; b \in B_i\} \cup S \cup S'$,
and its hyperarcs are:
\begin{itemize}
  \item[-] $( \{(i,a)\} \times B_i,\{\pi(i)\} ), \; i \in S, \; a \in A_i$;
  \item[-] $( \{j\}, \{(i,a,b)\} )$ for all $j,i \in S$, $a \in A_i$, $b \in B_i$, such that $(P_{i}^{ab})_j>0$.
\end{itemize}
Again, $G^-$ consists of two layers (see Figure~\ref{fig:HyperG-}).

\tikzset{->-/.style={decoration={markings,mark= at position 0.42 with {\arrow{#1}},},postaction={decorate}}}
\begin{figure}[htbp]
  \begin{center}
    \begin{tikzpicture}
      [node distance=0.75cm, on grid, bend angle=20, auto,
        state/.style={circle,draw,inner sep=0pt,minimum size=1.5mm},
        dots/.style={-,shorten >=3pt,shorten <=3pt,loosely dotted},
        simpledge/.style={->-={angle 45 reversed}},
        hyperedge/.style={>=angle 45},
        frame/.style={draw,rectangle,rounded corners=5pt}]
      \node[state] (s_1) at (0,0) {};
      \node[state] (s_2) [below=of s_1] {};
      \node[state] (s_3) [below=of s_2, label={[shift={(-0.2,-0.15)}]$j$}] {};
      \node[state] (s_4) [below=of s_3, yshift=-0.75cm] {}
        edge [dots] (s_3);
      \node[state] (s_5) [below=of s_4] {};
      \node[state] (a_1) [right=of s_2, xshift=2cm, label={[shift={(-0.05,-0.05)}]${(i,a,b')}$}] {}
        edge [simpledge] (s_1)
        edge [simpledge] (s_2)
        edge [simpledge] (s_3);
      \node[state] (a_2) [above=of a_1, yshift=0.75cm, label={[shift={(0.1,-0.05)}]${(i,a,b)}$}] {};
      \node (a_3) [above=of a_1, yshift=-0.45cm] {}
        edge [dots] (a_2);
      \node[state] (b_1) [right=of s_4, xshift=2cm,label={[shift={(-0.1,-0.75)}]${(i,a',b)}$}] {}
        edge [simpledge] (s_2)
        edge [simpledge] (s_4)
        edge [simpledge] (s_5);
      \node[state] (b_2) [below=of b_1, yshift=-0.75cm, label={[shift={(0.1,-0.75)}]${(i,a',b')}$}] {};
      \node (b_3) [below=of b_1, yshift=0.45cm] {}
        edge [dots] (b_2);
      \node[state] (s_6) [right=of s_3, xshift=4.75cm, yshift=-0.375cm, label={[shift={(0.4,-0.15)}]${\pi(i)}$}] {};
      \node[state] (s_7) [above=of s_6, yshift=0.75cm] {}
        edge [dots] (s_6);
      \node[state] (s_8) [below=of s_6, yshift=-0.75cm] {}
        edge [dots] (s_6);
      \hyperedgewithangles[0.3][$(hyper@tail)!0.55!(hyper@head)$]{a_2/-60,a_1/10}{-30}{s_6/145};
      \hyperedgewithangles[0.45][$(hyper@tail)!0.55!(hyper@head)$]{b_1/-10,b_2/60}{30}{s_6/-145};
      \node[frame,minimum height=4.55cm,minimum width=1.5cm] [below=of s_2,yshift=-0.375cm,label=above:$S$] {};
      \node[frame,minimum height=3cm,minimum width=1.5cm] [above=of a_1,label=above:${\{(i,a)\}\times B_i}$] {};
      \node[frame,minimum height=3cm,minimum width=1.5cm] [below=of b_1,label=below:${\{(i,a')\}\times B_i}$] {};
      \node[frame,minimum height=3.8cm,minimum width=1.5cm] [right=of s_3,xshift=4.75cm,yshift=-0.375cm,label=above:${S'=\pi(S)}$] {};
    \end{tikzpicture}
  \end{center}
  \caption{Hypergraph $G^-$ associated with $F^-$}
  \label{fig:HyperG-}
\end{figure}
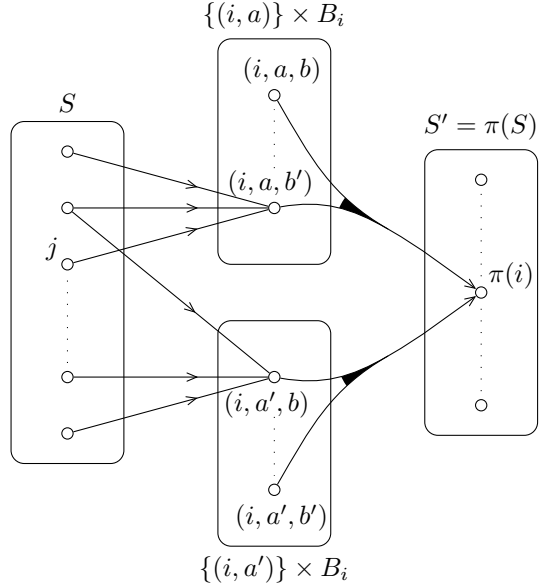

Like $G^+$, the motivation for the construction of $G^-$ is the following result.
\begin{proposition}
  \label{prop:HypergraphG-}
  Let $F$ be the payment-free Shapley operator associated with the actions spaces $A_i$ and $B_i$ of the two players, and the transition probability $P$, and let $F^-$ be defined by~\eqref{defF-}.  
  Then, the node $\pi(i) \in S'$ is reachable from $I \subset S$ in $G^-$ if, and only if, $[F^-(\indic_\Cpt{I})]_i = 0$. \qed
\end{proposition}

So far we did not make any assumption about the action spaces, which may be infinite, leading to infinite hypergraphs $G^+$ and $G^-$.
The absence of symmetry between $G^+$ and $G^-$ reflects the lack of symmetry between $F^+$ and $F^-$.

Denote $\bar G^+$ and $\bar G^-$ the hypergraphs obtained from $G^+$ and $G^-$, respectively, by identifying every node $i\in S$ with node $\pi(i)\in S'$.
The following proposition is immediate.

\begin{proposition}
  \label{prop:CheckingH1H2}
  A subset $I\subset S$ belongs to $\F'^-$
  if, and only if, its complement in $S$ 
  is an invariant set in the hypergraph $\bar G^+$ relatively to $S$:
  $\Reach(S\setminus I, \bar{G}^+)\cap S = S \setminus I$.
  A subset $J\subset S$ belongs to $\F'^+$
  if, and only if, its complement in $S$ 
  is an invariant set in the hypergraph $\bar G^-$ relatively to $S$:
  $\Reach(S \setminus J, \bar{G}^-) \cap S = S \setminus J$.
  \qed 
\end{proposition}

\begin{corollary}
  \label{coro:ComputationPhi}
  Let $F$ be the payment-free Shapley operator
  associated with the actions spaces $A_i$ and $B_i$ of the two players, 
  and the transition probability $P$, and let $F^+$ and $F^-$
  be defined by~\eqref{defF+} and~\eqref{defF-} respectively.
  Let $I \in \F'^-$ and $J \in \F'^+$.
  Then $\Phi'(I)$ is given by the complement in $S$ of all the nodes of $S$ that are reachable from $I$ in $\bar{G}^-$. Moreover,
  $\Phi'^\star(J)$ is given by the complement in $S$ of all the nodes of $S$ that are reachable from $J$ in $\bar{G}^+$.
\end{corollary}

\begin{proof}
  It follows readily from the definition of $\Phi'$ (by~\eqref{eq:GaloisConnection} with  $\F'^-$ and $\F'^+$ instead of  $\F^-$ and $\F^+$), 
  that $S\setminus \Phi'(I)$ 
  is the smallest set $I'$ containing $I$ such that $S\setminus I' \in \F'^+$.
  By Proposition~\ref{prop:CheckingH1H2}, the latter condition holds
  if, and only if, $I'$ satisfies $\Reach(I',\bar{G}^-) \cap S = I'$. 
  Hence, $\Phi'(I)$ is the complement in $S$ of the set of nodes of $S$
  that are reachable from $I$ in $\bar{G}^-$.
  The argument for $\Phi'^\star$ is dual.
\end{proof}
We shall say that $I,J\subset S$ are {\em conjugate subsets of states}
with respect to the hypergraphs $\bar{G}^+,\bar{G}^-$
if $I,J$ are nonempty and if
\[
J= S \setminus (\Reach(I,\bar{G}^-)\cap S) \qquad \text{and} \qquad 
I = S \setminus (\Reach(J, \bar{G}^+)\cap S) \enspace .
\]

\begin{theorem}
  \label{th-conjugates}
  Let $F$ be the payment-free Shapley operator associated with actions spaces $A_i$ and $B_i$ of the two players, and a transition probability $P$, such that Assumption~\ref{StandardAssumptions} holds.
  Then the following assertions are equivalent: 
  \begin{enumerate}
    \item $F$  has a nontrivial fixed point;
    \item\label{th-conj-2} there exist nonempty disjoint subsets $I,J\subset S$ such that 
      $ S \setminus I$ is invariant in  $\bar{G}^+$ relatively to $S$, and
      $S\setminus J$ is invariant in $\bar{G}^-$ relatively to $S$;
    \item\label{th-conj-3} there exist conjugate subsets of states $I,J\subset S$
      with respect to the hypergraphs $\bar{G}^+,\bar{G}^-$.
  \end{enumerate}
  Furthermore, for any sets $I,J \subset S$, they are conjugate with respect to the Galois connection $(\Phi,\Phi^\star)$ if, and only if, they are conjugate with respect to the hypergraphs $\bar{G}^+,\bar{G}^-$.
\end{theorem}

\begin{proof}
  Recall that, under Assumption~\ref{StandardAssumptions},
  $\F'^-$, $\F'^+$, $\Phi'$ and $\Phi'^\star$ 
  coincide with $\F^-$, $\F^+$, $\Phi$ and $\Phi^\star$ respectively.
  Then, the theorem follows from Theorem~\ref{thm:Galois} and from the characterization
  of the Galois connection $(\Phi',\Phi'^*)$ in terms of hypergraph
  reachability given in Corollary~\ref{coro:ComputationPhi}.
\end{proof}

\begin{remark}\label{remark-hypergraph-graph}
It is instructive to specialize the latter result to the case in which each player has only one possible action in each state. Then, we can write $F(x)=Px$, where $P$ is a stochastic matrix.
The two hypergraphs $\bar{G}^+$ and $\bar{G}^-$ 
are isomorphic (up to the identification of $(i,a)$ and $(i,a,b)$ with $\pi(i)$)
to the transpose of the digraph $G$ associated to $P$ 
(the arcs of $\bar{G}^+$ and $\bar{G}^-$ are in the opposite direction).
In particular, $ S \setminus I$ is invariant in  $\bar{G}^+$ (or $\bar{G}^-$)
 relatively to $S$ if and only if there are no arcs or paths from $I$ to
 $ S \setminus I$ in $G$.
Similarly, $(I,J)$ is a pair of conjugate subsets of states
with respect to $(\bar{G}^+,\bar{G}^-)$ if and only if 
$J$ is the greatest set of nodes with no paths in $G$ to a node of $I$,
and vice versa.
So Assertion~\pref{th-conj-2} of Theorem~\ref{th-conjugates}
implies the existence 
of two distinct final classes in $G$: $I$ and $J$ are disjoint and
there are no arcs from $I$ to $ S \setminus I$,
and similarly there are no arcs from $J$ to $ S \setminus J$,
so there exists a final class of $G$ included in $I$ and also a final class
of $G$ included in $J$, and since $I$ and $J$ are disjoint, there are
two distinct final classes of $G$.
Moreover, if $I$ and $J$ are two distinct final classes, then
there are no arcs or paths from $I$ to  $ S \setminus I$ in $G$.
The same is true for $J$ so that Assertion~\pref{th-conj-2} of
Theorem~\ref{th-conjugates} holds.
Hence in the present case, Assertion~\pref{th-conj-2} of 
Theorem~\ref{th-conjugates} corresponds to the condition that
the digraph associated to $P$ has two distinct final classes, that is the
opposite of Assertion~\pref{prop-ergodic-graph} in 
Theorem~\ref{thm:ErgodicityMarkovChain}.
\end{remark}

\section{Algorithmic issues}
\label{algorithmic-section}

Let us fix a state space $S=[n]$, and the nonempty finite actions spaces $A_i$ and $B_i$ of the two players.
We denote by $m_1$ the number of couples $(i,a)$ with $i \in S$ and $a \in A_i$,
and by $m_2$ the number of triples $(i,a,b)$ with $i \in S$, $a \in A_i$ and $b \in B_i$.
Then $n \leq m_1 \leq n m$ and $m_1 \leq m_2\leq n m^2$
where $m$ is the greatest cardinality of $A_i$ and $B_i, \enspace i \in S$,
and we have $m_1 \leq \size(G^+) = O(n m_1) \leq O(n^2 m)$
and $m_2 \leq \size(G^-) = O(n m_2) \leq O(n^2 m^2)$.

\subsection{Checking ergodicity}
\label{sec:CheckingErgodicity}

From Theorem~\ref{th-game-ergodicity-full}, the negation of the following problem is equivalent to the next one. 

\begin{problem}[\NonTrivialFP]
  \label{pb:NonTrivialFP}
  Does a given payment-free Shapley operator $F:\R^n \to \R^n$ 
  with finite action spaces have a non trivial fixed point, that is, does there exist $u \in \R^n \setminus \R \, \unit$ such that $u = F(u)$?
\end{problem}

\begin{problem}[\Ergodicity]
  \label{pb:ergodicity}
  Is a given game $\Gamma(r,P)$ with finite action spaces
  and bounded  payment $r$ ergodic?
\end{problem}

It is known that in a directed hypergraph $G$, the set of reachable nodes from a set $I$
can be computed in $O(\size(G))$ time~\cite{GLNP93}. Hence,
the following result follows from Proposition~\ref{prop:CheckingH1H2} and 
Corollary~\ref{coro:ComputationPhi}
and the property that, when the actions spaces are finite,
$\F'^-$, $\F'^+$, $\Phi'$ and $\Phi'^\star$ 
coincide with $\F^-$, $\F^+$, $\Phi$ and $\Phi^\star$ respectively.

\begin{proposition}
  \label{computephi}
  Let us fix a state space $S=[n]$, and the nonempty finite actions
  spaces $A_i$ and $B_i$ of the two players.
  For any payment-free Shapley operator $F$, and any $I,J\subset S$,
checking that $I\in \F^-$ and $J \in \F^+$ can be done 
  respectively in $O(nm_1)\leq O(n^2 m)$ and $O(nm_2)\leq O(n^2 m^2)$ time.
Moreover, for any $I \in \F^-$ and $J \in \F^+$,
  $\Phi(I)$ and $\Phi^\star(J)$ can be evaluated
  respectively in $O(nm_2)\leq O(n^2 m^2)$ and $O(nm_1)\leq O(n^2 m)$ time.
\end{proposition}

Using Proposition~\ref{computephi} and Theorem~\ref{th-conjugates},
we obtain the following result,
which shows that checking the ergodicity of a game
is fixed parameter tractable: 
if the dimension is fixed, we can solve it in a time which is polynomial in the input-size. 
Thus, for instances of moderate dimension, but with large action spaces, the ergodicity condition can be checked efficiently.

\begin{theorem}\label{th-fptrac}
  Let us fix a state space $S=[n]$, and the nonempty finite actions
  spaces $A_i$ and $B_i$ of the two players.
  Let $r$ be a bounded transition payment and $P$ be a transition probability.
  Then, the ergodicity of $\Gamma(r,P)$,
  that is the property ``$\hat{T}$ has only trivial fixed points'', can
  be checked in $O(2^n nm_2) \leq O(2^nn^2m^2)$ time.
\end{theorem}

Problem~\NonTrivialFP\ has already been addressed in the deterministic case with finite action spaces by Yang an Zhao~\cite{YZ04}.
Suppose indeed that in the expression~\eqref{eq:PaymentFree}, the support of each transition probability is concentrated on just one state and consider the restriction of such an operator to the Boolean vectors $\{0,1\}^n$.
We obtain a monotone Boolean operator.

Recall that a Boolean operator, defined on Boolean vectors $\{0,1\}^n$, is expressed using the logical operators \textsf{AND}, \textsf{OR} and \textsf{NOT}.
Monotone Boolean operators are those whose expression involves only \textsf{AND} and \textsf{OR} operators.
These can be interpreted as $\min$ and $\max$ operators, respectively. So,
deterministic payment-free Shapley operators are equivalent
to monotone Boolean operators and Problem~\NonTrivialFP\ can be expressed in a simpler form.

\begin{problem}[\MonBool]
  \label{pb:MonBool}
  Does a given monotone Boolean operator have a nontrivial fixed point, that is, different from the zero vector and the unit vector?
\end{problem}

\begin{theorem}[Yang, Zhao~\cite{YZ04}]
  \label{thm:MonBool}
  Problem~\MonBool\ is NP-complete.
\end{theorem}

Using this result and the characterizations of the previous section, we obtain:
\begin{corollary}
  \label{thm:NonTrivialFP}
  Problem~\NonTrivialFP\ is NP-complete.
\end{corollary}

\begin{proof}
  As a direct consequence of Theorem~\ref{thm:MonBool}, we get that Problem~\NonTrivialFP\ is NP-hard.
  We now show that it is in NP.
  Suppose that a payment-free Shapley operator $F$ has a nontrivial fixed point $u$.
  Then  $\argmin u$ and $\argmax u$ are proper subsets of states, and by Lemma~\ref{lem:NontrivialFPandGaloisConnection}, $\argmin u\in\F^-$ and $\Phi(\argmin u)\supset \argmax u\neq \emptyset$. Hence, $\argmin u\in\F^-$ 
  is a proper subset of states such that  $\Phi(\argmin u)$ is nonempty, and
  we know by Theorem~\ref{thm:Galois} that these conditions are sufficient to guarantee the existence of a nontrivial fixed point.
  Furthermore, these conditions can be checked in polynomial time (this is a consequence of Proposition~\ref{computephi}).
  Hence, $\argmin u$ is a short certificate to Problem~\NonTrivialFP.
\end{proof}

\subsection{Problem \IisMin}

A way to analyze Problem~\NonTrivialFP\ would be to characterize the fixed point set $\mathcal{W} := \{ w \in \R^n \mid F(w)=w \}$ of a payment-free Shapley operator $F$.
This problem also arises in several other situations.
First in Proposition~\ref{prop:GivenMeanPayoff}, we have shown that $\mathcal{W}$ is exactly the set of possible mean payoff vectors of the game $\Gamma(r,P)$ when the transition payment $r$ varies.
Next, in~\cite{Eve57}, Everett introduced the notion of recursive games which are modified versions of the game  $\Gamma(0,P)$ in which payments occur in some absorbing states.
These games are well posed if there exists a unique element of $\mathcal{W}$ with prescribed values in the absorbing states.
Finally, $\mathcal{W}$ allows one to determine the set $\mathcal{E}$ of solutions $u$ of the
ergodic equation $T(u)=\lambda\unit +u$.
Indeed, it is shown in~\cite{AGN12} that if the Shapley operator $T$ is piecewise affine, if $u$ is any point in $\mathcal{E}$, and if $\mathcal{V}$ is a neighborhood of $0$, then, $\mathcal{E} \cap (u+\mathcal{V}) = u+ \{ w \in \mathcal{V} \mid F(w)=w \} = u + (\mathcal{V} \cap \mathcal{W})$, where $F$ is a payment-free Shapley operator (the semidifferential of $T$ at point $u$).
Hence, the local study of the ergodic equation reduces to the characterization of the fixed point
set $\mathcal{W}$.

In an attempt to understand the structure of the set of
fixed points of a payment-free Shapley operator, we shall consider the
following simpler problem.
\begin{problem}[\IisMin]
  \label{pb:I=Min}
  Let $I$ be a subset of $S$.
  Does a given payment-free Shapley operator with finite action spaces have a fixed point $u$ satisfying $I = \argmin u$?
\end{problem}

We know from Lemma~\ref{lem:NontrivialFPandGaloisConnection} that a necessary condition is $I \in \F^-$.
Under Assumption~\ref{StandardAssumptions} (which is the case if action spaces are finite), this is equivalent to $F^+(\indic_\Cpt{I}) \leq \indic_\Cpt{I}$.
In fact, there is a stronger necessary condition.
\begin{lemma}
  \label{lem:NecessaryCondition}
  Let $F$ be the payment-free Shapley operator associated with actions spaces $A_i$ and $B_i$ of the two players, and a transition probability $P$, such that Assumption~\ref{StandardAssumptions} holds
  and let $I \subset S$.
  Suppose that $F$ has a fixed point $u$ verifying $\argmin u = I$.
  Then, $F^+(\indic_\Cpt{I}) = \indic_\Cpt{I}$.
\end{lemma}

\begin{proof}
  If $I = S$, the conclusion of the lemma is trivial.
  Assume $I\neq S$ and let $u$ be a fixed point of $F$ verifying $\argmin u = I$.
  We may suppose that $\min_{i \in S} u_i = 0$ and $\max_{i \in S} u_i = 1$, so that $u \leq \indic_\Cpt{I}$.
  Since $F \leq F^+$, we get $u=F(u)\leq F^+(u) \leq F^+(\indic_\Cpt{I})$.
  The last vector is Boolean, so this inequality implies $\indic_\Cpt{I} \leq F^+(\indic_\Cpt{I})$.
  Moreover, according to Lemma~\ref{lem:NontrivialFPandGaloisConnection} and Lemma~\ref{lem:BooleanReformulation}, we already know that $F^+(\indic_\Cpt{I}) \leq \indic_\Cpt{I}$.
  Hence the result.
\end{proof}

We continue with another necessary condition.
\begin{lemma}
  \label{lem:EmptyPhi}
  Let $F$ be the payment-free Shapley operator associated with actions spaces $A_i$ and $B_i$ of the two players, and a transition probability $P$,
  and let $I \in \F^-$.
  If $\Phi(I) = \emptyset$, then $F$ has no nontrivial fixed point $u$ satisfying $I \subset \argmin u$.
\end{lemma}

\begin{proof}
  Suppose on the contrary that there is a nontrivial fixed point $u$ such that $I \subset \argmin u$.
  Let $I' = \argmin u$ and $J := \argmax u$.
  We know from Lemma~\ref{lem:NontrivialFPandGaloisConnection} that $I' \in \F^-$, $J \in \F^+$ and that $J \subset \Phi(I')$.
  Since $I \subset I'$, we have $\Phi(I') \subset  \Phi(I)$.
  Hence $J \subset \Phi(I)$, and since $J \neq \emptyset$, we get
  a contradiction.
\end{proof}

If $I=\emptyset$, the answer to Problem~\IisMin\ is trivially negative,
and if $I=S$ it is trivially positive. 
Assume now that $I$ is a proper subset of $S$.
The above results show that
a necessary condition to have a positive answer to problem~\IisMin\ 
is that $I \in \F^-$  and  $\Phi(I) \neq \emptyset$.
Moreover, by Lemma~\ref{lem:ArgmaxPhi}, a sufficient condition to have a positive answer to problem~\IisMin\ is that $I$ is closed with respect to the Galois connection $(\Phi, \Phi^\star)$.

It remains to examine the case in which $I \in \F^-$ is proper,
with $\Phi(I) \neq \emptyset$ and $I \neq \clo{I}$, where for $I\in \F^-$, $\clo{I} := \Phi^\star (\Phi(I))$ denotes the closure of $I$ with respect to the Galois connection $(\Phi, \Phi^\star)$ (likewise, for $J \in \F^+$, $\clo{J}$ is the closure of $J$ with respect to the Galois connection $(\Phi^\star, \Phi)$).
This implies in particular that $\clo{I} \neq S$ (otherwise we would have $\Phi(I) = \Phi(\clo{I}) = \emptyset$).

Assume that Assumption~\ref{StandardAssumptions} holds.
We define a reduced operator $F^\vartriangle: \R^\clo{I} \to \R^\clo{I}$ as follows.
According to the game-theoretic interpretation (Proposition~\ref{prop:GameInterpretation}), we know that player \MIN\ can force the state of the game $\Gamma(0,P)$ (which has $F$ as Shapley operator) to stay in $\clo{I}$.
Hence, we consider the actions of player \MIN\ that achieve this goal: for every $i \in \clo{I}$, let
\[ A^\vartriangle_i := \{ a \in A_i \mid \forall b \in B_i\; \text{and} \; j\in \Cpt{\clo{I}},\; (P_i^{a b})_j =0 \} \enspace . \]
These sets are nonempty, since $\clo{I} \in \F^-=\F'^-$.
Another formulation of $A^\vartriangle_i$ is the following:
\[ A^\vartriangle_i := \{ a \in A_i \mid \max_{b \in B_i} P_i^{a b} \indic_\Cpt{\clo{I}} = [F(\indic_\Cpt{\clo{I}})]_i = 0 \} \enspace . \]
For $x \in \R^n$ and $K \subset S$, we denote by $x_K$ the restriction of $x$ to $\R^K$. We apply the same notation to elements of $\Delta(S)$.
Then, for every $i \in \clo{I}$, let
\[ [ F^\vartriangle(x) ]_i := \min_{a \in A^\vartriangle_i} \max_{b \in B_i} (P_i^{a b})_{\clo{I}} \, x, \quad x \in \R^{\clo{I}} \enspace . \]
From the definition of $A^\vartriangle_i$, we have that $(P_i^{a b})_{\clo{I}}\in\Delta(\clo{I})$ for all $i\in \clo{I}$, $a\in  A^\vartriangle_i$ and $b\in B_i$.
Hence, $F^\vartriangle$ is a payment-free Shapley operator over $\clo{I}$,
with actions spaces $A^\vartriangle_i$ and $B_i$ and transition probability
$P^\vartriangle:(i,a,b)\mapsto (P_i^{a b})_{\clo{I}}$.
Moreover, we have
\begin{equation}
  \label{proptriangle}
  [ F^\vartriangle(x_{\clo{I}}) ]_i =\min_{a \in A^\vartriangle_i} \max_{b \in B_i} P_i^{a b} x, \quad i \in\clo{I},\;  x\in  \R^n\enspace. 
\end{equation}

\begin{theorem}
  \label{thm:Reduction}
  Let $F$ be the payment-free Shapley operator associated with finite actions spaces $A_i$ and $B_i$ of the two players, and a transition probability $P$.
  Let $I \in \F^-$ be proper, such 
  that $\Phi(I) \neq \emptyset$ and $I \neq \clo{I}$.
  Then $F$ has a fixed point whose $\argmin$ is $I$ if, and only if, the same holds for the reduced operator $F^\vartriangle$.
\end{theorem}

\begin{proof}
  We first show the ``only if'' part of the theorem.
  Let $u$ be a fixed point of $F$ such that $I = \argmin u$.
  Recall that $I \neq S$ by hypothesis.
  So we may suppose
  that $\max_{i \in S} u_i = 1$ and $\min_{i \in S} u_i = 0$.
  
  It follows from~\eqref{proptriangle} that $[F(u)]_\clo{I} \leq F^\vartriangle(u_\clo{I})$.
  Hence $u_\clo{I} = [F(u)]_\clo{I}\leq F^\vartriangle(u_\clo{I})$,
  so that  $(F^\vartriangle)^\omega(u_\clo{I})$ exists.
  Let us denote it by $v$.
  It is a fixed point of $F^\vartriangle$ and it satisfies $u_\clo{I} \leq v$.
  As a consequence, $v_i > 0$ for every $i \in \clo{I} \setminus I$.
  
  Furthermore, Lemma~\ref{lem:NontrivialFP} implies that $I\in \F^-$,
  meaning that $F(\indic_{S \setminus I})\leq \indic_{S \setminus I}$.
  Then, for all $i\in I$, there exists $a\in A_i$ such that for all
  $b\in B_i$, $P_i^{ab} \indic_{S \setminus I}=0$.
  Since $I\subset \clo{I}$, this implies that $(P_i^{ab})_j=0$ for all
  $j\in S\setminus\clo{I}$,
  and since this holds for all $b\in B_i$, we deduce that $a\in A_i^\vartriangle$,
  by definition. Hence, 
  $\min_{a\in A_i^\vartriangle} \max_{b\in B_i} P_i^{ab}\indic_{S \setminus I}= 0$,
  and using~\eqref{proptriangle}, we deduce that
  $F^\vartriangle(\indic_{\clo{I} \setminus I})=0$ for all $i \in I$.
  Therefore  $F^\vartriangle (\indic_{\clo{I} \setminus I}) \leq \indic_{\clo{I}\setminus I}$,
  which means that $I$ still satisfies condition~\eqref{eq:H1} with the operator $F^\vartriangle$.
  Since $u_\clo{I} \leq \unit_{\clo{I} \setminus I}$, it follows that 
  $v=(F^\vartriangle)^\omega(u_\clo{I})\leq (F^\vartriangle)^\omega (\unit_{\clo{I} \setminus I})
  \leq \unit_{\clo{I} \setminus I}$. Hence $v_i = 0$ for every $i \in I$,
  which shows that $\argmin v=I$.
  
  We now prove the ``if'' part of the theorem.
  Assume that $F^\vartriangle$ has a fixed point $v$ such that $\argmin v = I$.
  We may suppose that $\max_{i \in S} v_i = 1$ and $\min_{i \in S} v_i = 0$.
  
  Let $w = F^\omega(\indic_\Cpt{\clo{I}})$.
  We know from Lemma~\ref{lem:ArgmaxPhi} that $w$ is a fixed point of $F$ such that $\argmin w = \clo I$.
  Thus, it satisfies $w_\clo{I} = \zero$ and $w_s >0$ for every $s \in \Cpt{\clo{I}}$, hence $w \geq\alpha\indic_\Cpt{\clo{I}}$ for some $\alpha>0$.
  
  We next use the notions of semidifferentiability and semiderivative, referring the reader to~\cite{RW98,AGN12} for the definition of these notions
  and for their basic properties.
  Since the action spaces are finite, $F$ is piecewise affine and so it is semidifferentiable at point $w$.
  Furthermore, denoting $F'_w$ its semiderivative at $w$, there is a neighborhood $\mathcal V$ of $\zero$ such that
  \begin{equation}
    \label{eq:ExactFirstOrder}
    F(w+x) = F(w) + F'_w(x), \quad \forall x \in \mathcal{V} \enspace .
  \end{equation}
  We next give a formula for $F'_w$.
  For every $i \in S$, let
  \[ A_i(w) := \big\{ a \in A_i \mid \max_{b \in B_i} P_i^{a b} w = [F(w)]_i \big\} \]
  and for $a \in A_i(w)$, let
  \[ B_i^a(w) := \big\{ b \in B_i \mid P_i^{a b} w = [F(w)]_i \big\} \enspace . \]
  Then we have, for every $x \in \R^n$ and every $i \in S$,
  \[ \big[ F'_w(x) \big]_i = \min_{a \in A_i(w)} \max_{b \in B_i^a(w)} P_i^{a b} x \enspace . \]
  
  Observe that for $i \in \clo{I}$, we have $A_i(w) = A_i^\vartriangle$ and $B_i^a(w) = B_i$, for every $a \in A_i(w)$.
  This is because $[F(w)]_i = w_i = 0$ and $\alpha\indic_\Cpt{\clo{I}}
  \leq w \leq\indic_\Cpt{\clo{I}}$, then $a\in  A_i(w)$ if and only
  if  $\max_{b \in B_i} P_i^{a b} \indic_\Cpt{\clo{I}}=0$ and
  $b\in B_i^a(w)$ if and only if $P_i^{a b} \indic_\Cpt{\clo{I}}=0$.
  Then, using~\eqref{proptriangle}, 
  we obtain $[F'_w(x)]_\clo{I} = F^\vartriangle(x_\clo{I})$ for every $x \in \R^n$.
  
  We introduce now the vector $z \in [0,1]^n$ given by $z_\clo{I} = v$ 
  and $z_{\Cpt{\clo{I}}} = 0$.
  By the above property of $F'_w$, we get that $[F'_w(z)]_\clo{I} 
  =F^\vartriangle(v)=v=z_\clo{I}$. Moreover,
  since $F'_w$ is a payment-free operator, and $z\geq 0$, we 
  get that  $F'_w(z) \geq 0$, so $F'_w(z) \geq z$.
  Hence, $\bar{z} = (F'_w)^\omega(z)$ exists and is a fixed point of $F'_w$, belonging to $[0,1]^n$.
  Again by the above property of $F'_w$, we get that $[(F'_w)^k(z)]_\clo{I} 
  =F^\vartriangle([(F'_w)^{k-1}(z)]_\clo{I} )$ for all $k\geq 1$, so that
  by induction $[(F'_w)^k(z)]_\clo{I}=v$, and $\bar{z}_\clo{I}=v$.
  
  Choose $\varepsilon > 0$ small enough so that $\varepsilon \bar{z}$ is in $\mathcal{V}$ and let $u = w+\varepsilon \, \bar{z}$.
  Then, from~\eqref{eq:ExactFirstOrder}, we get that 
  $F(u)=F(w)+\varepsilon F'_{w}(\bar{z})=w+\varepsilon \bar{z}=u$, where
  we used the fact that $F'_w$ is positively homogeneous.
  Then $u$ is a fixed point of $F$.
  Moreover, by construction $u= w+\varepsilon \, \bar{z}\geq w$
  and $u\geq \epsilon \bar{z}$,
  and since $\argmin w=\clo{I}$ and $\argmin \bar{z}\cap \clo{I}=I$,
  we deduce that $u_I = \zero$ and $u_s > 0$ for every 
  $s \in \Cpt{I}$, that is $\argmin u = I$.
\end{proof}

The previous result together with the observations made before lead
to Algorithm~\ref{algo1} 
below, which solves Problem~\IisMin, as detailed in Theorem~\ref{thm:I=Min}.
There, we are still assuming that for each state $i \in S$ the action spaces $A_i$ and $B_i$ are finite.
Moreover, if $F$ is a payment-free Shapley operator, we write $(\Phi_F, \Phi^\star_F)$ the Galois connection associated to that operator.
\begin{algorithm}[htbp]
  \caption{}
  \label{algo1}
  \begin{algorithmic}[1]
    \REQUIRE $S$, $A_i$, $B_i$, $P$, the corresponding 
    payment-free Shapley operator $F: \R^S \to \R^S$ and $I \subset S$
    \ENSURE answer to Problem~\IisMin
    \IF{ $I=\emptyset$ } \RETURN \FALSE
    \ELSIF{  $I=S$ } \RETURN \TRUE
    \ELSE
      \LOOP
        \IF{ $F^+(\indic_\Cpt{I}) \neq \indic_\Cpt{I}$ \OR $\Phi_F(I) = \emptyset$} \RETURN \FALSE
        \ELSIF{$\Phi^\star_F (\Phi_F(I)) = I$} \RETURN \TRUE
        \ELSE \STATE $A_i \leftarrow A_i^\vartriangle$, $P\leftarrow P^\vartriangle$, $F \leftarrow F^\vartriangle$, $S\leftarrow \Phi^\star_F (\Phi_F(I))$
        \ENDIF
      \ENDLOOP
    \ENDIF
  \end{algorithmic}
\end{algorithm}

\begin{theorem}
  \label{thm:I=Min}
  Algorithm~\ref{algo1} solves Problem~\IisMin\ in $O(n^2 m_2) \leq O(n^3 m^2)$ time.
\end{theorem}

\begin{proof}
  The fact that Algorithm~\ref{algo1} provides the right answer is a direct consequence of Lemma~\ref{lem:NecessaryCondition}, Lemma~\ref{lem:EmptyPhi}, Lemma~\ref{lem:ArgmaxPhi} and Theorem~\ref{thm:Reduction}.
  
  We next show that it stops after at most $n$ iterations of the loop.
  Suppose that during the execution of a loop, the first two conditions (which are stopping criteria) are not satisfied.
  Then the closure of $I$ with respect to the Galois connection $(\Phi_F, \Phi^\star_F)$ associated with $F$ is a proper subset of states.
  Hence, the cardinality of the state space for
  the reduced operator $F^\vartriangle$
  is strictly less than the one of $F$.
  
  Moreover, each operation in the loop requires at most $O(n m_2) \leq O(n^2 m^2)$ time (see Proposition~\ref{computephi}).
\end{proof}

\subsection{Mixed problem}

So far, we have only considered the problem with a single constraint on the fixed point, concerning the indices of the minimal entries.
The dual problem, concerning the maximal entries of fixed points, is equivalent.
We address now a mixed-condition problem.
\begin{problem}[{\bf IMinJMax}]
  \label{pb:IMinJmax}
  Let $I$ and $J$ be nonempty disjoint subsets of $S$.
  Does a given payment-free Shapley operator with finite action spaces have a fixed point $u$ satisfying $I = \argmin u$ and $J = \argmax u$?
\end{problem}

Let $F$ be a payment-free Shapley operator with finite action spaces and let $I,J$ be two nonempty disjoint subsets of $S$.
We already know from Lemma~\ref{lem:NecessaryCondition} and its dual formulation that $F^+(\indic_\Cpt{I})=\indic_\Cpt{I}$ and $F^-(\indic_J)=\indic_J$ are necessary conditions to have a positive answer to problem~\IMinJMax.
The following theorem shows that the two constraints can be treated separately.
\begin{theorem}
  \label{thm:IMinJMax}
  Let $F$ be the payment-free Shapley operator associated with actions spaces $A_i$ and $B_i$ of the two players, and a transition probability $P$.
  Let $I \in \F^-$ and $J \in \F^+$ be two nonempty disjoint subsets.
  Then $F$ has a fixed point $u$ satisfying $I = \argmin u$ and $J = \argmax u$ if and only if $F$ has fixed points $v, w$ satisfying $\argmin v = I$ and $\argmax w = J$.
\end{theorem}

\begin{proof}
  We only  need to prove the ``if'' part of the theorem.
  Suppose that $F$ has fixed points $v, w$ satisfying $\argmin v = I$ and $\argmax w = J$.
  Then, we may impose $\min_{i \in S} v_i = 0$, $\max_{i \in S} v_i = \min_{i \in S} w_i = 1/2$ and $\max_{i \in S} w_i = 1$.
  
  Let $\mathcal L=\{z \in \R^n \mid v \vee \indic_J \leq z \leq w \wedge \indic_\Cpt{I} \}$.
  Put in words, $\mathcal L$ is the set of all elements in $[0,1]^n$ whose entries are $0$ on $I$, $1$ on $J$ and comprised between those of $v$ and $w$ elsewhere.
  In particular, the entries outside $I$ or $J$ of the elements in $\mathcal L$ are in $(0,1)$.
  
  The set $\mathcal L$ is a complete lattice.
  Since $J \in \F^+$, we have $v \vee \indic_J \leq F(v) \vee F(\indic_J) \leq F(v \vee \indic_J)$.
  Since $I \in \F^-$, we have $w \wedge \indic_\Cpt{I} \geq F(w) \wedge F(\indic_\Cpt{I}) \geq F(w \wedge \indic_\Cpt{I})$.
  Hence, $v \vee \indic_J \leq z \leq w \wedge \indic_\Cpt{I} $ implies 
  $v \vee \indic_J \leq F(v \vee \indic_J)\leq F( z) \leq F(w \wedge \indic_\Cpt{I})\leq w \wedge \indic_\Cpt{I} $, which shows that $\mathcal L$ is invariant by $F$.
  As $F$ is order-preserving, Tarski's fixed point theorem guarantees the existence of a fixed point of $F$ in $\mathcal L$.
\end{proof}

\begin{corollary}
  \label{coro:IMinJMax}
  Problem~\IMinJMax\ can be solved in $O(n^2 m_2) \leq O(n^3 m^2)$ time.
\end{corollary}

\begin{proof}
  According to Theorem~\ref{thm:IMinJMax}, Problem~\IMinJMax\ can be solved by two instances of Problem~\IisMin, one with inputs $F$ and $I$, one with inputs $\widetilde F$ and $J$.
\end{proof}

\subsection{Summary of complexity results}

The following table summarizes the results of this section.
\begin{center}
  \begin{tabular}{l|l}
    Problem & Complexity class\\
    \hline
    \MonBool & NP-complete (\cite{YZ04})\\
    \hline
    \NonTrivialFP & NP-complete (Corollary~\ref{thm:NonTrivialFP})\\
    \hline
    \IisMin & P (Theorem~\ref{thm:I=Min})\\
    \hline
    \IMinJMax & P (Corollary~\ref{coro:IMinJMax})\\
  \end{tabular}
\end{center}

\section{Example}
\label{ex-section}

\subsection{Checking ergodicity}

We consider the game with perfect information defined by the graph represented in Figure~\ref{fig:ExampleFullGraph}.
There are four states represented by gray nodes.
A token is initially placed in one of these nodes.
At each stage, the token is moved along the edges of the graph until it reaches another state, according to the following rule: player \MIN\ moves the token at circle nodes, player \MAX\ at square ones and at the diamond nodes, an edge is selected at random according to the probabilities indicated on the edges
starting from the node.
A payment occurs only for the edges starting from a \MAX\ node (its value is given by the label attached to such edges).

\begin{figure}[htbp]
  \begin{center}
    \begin{tikzpicture}
      [node distance=1.6cm, on grid, bend angle=20, auto,
        min/.style={circle,draw,minimum size=7mm,fill=black!15},
        max/.style={rectangle,draw,minimum size=5mm},
        nature/.style={diamond,draw,inner sep=0pt,minimum size=6mm},
        tail/.style={->,shorten >=3pt,>=angle 60}, head/.style={<-,shorten <=3pt,>=angle 60}]
      \node[min] (min1) at (0,0) {\large \bf 1};
      \node[max] (max1) [left=of min1, xshift=-0.8cm] {}
        edge [head] (min1);
      \node[nature] (n1) [below=of max1] {}
        edge [tail] node[midway,sloped,above]{\scriptsize $1/2$} (min1)
        edge [head] node[left,\colorpayment]{$1$} (max1);
      \node[max] (max2) [right=of n1] {}
        edge [tail] node[below,\colorpayment]{$-2$} (n1)
        edge [tail] node[left,\colorpayment]{$-1$} (min1);
      \node[min] (min2) [below=of max2, xshift=0.8cm] {\large \bf 2}
        edge [head, bend left] node[midway,sloped,below]{\scriptsize $1/2$} (n1)
        edge [tail] (max2);
      \node[max] (max3) [right=of max2] {}
        edge [head, bend left] (min1)
        edge [tail, bend right] node[right,\colorpayment]{$2$} (min1);
      \node[max] (max4) [right=of max3] {}
        edge [head] (min2);
      \node[nature] (n2) [above=of max4] {}
        edge [head] node[left,\colorpayment]{$2$} (max4)
        edge [tail] node[midway,above]{\scriptsize $1/2$} (min1);
      \node[max] (max5) [right=of max4] {}
        edge [tail] node[below,,xshift=-0.2cm,\colorpayment]{$-3$} (n2);
      \node[min] (min3) [right=of n2, xshift=0.8cm] {\large \bf 3}
        edge [tail] (max5)
        edge [head] node[midway,above]{\scriptsize $1/2$} (n2);
      \node[max] (max6) [right=of max5] {};
      \node[nature] (n3) [right=of max6] {}
        edge [head] node[below,\colorpayment]{$2$} (max6)
        edge [tail, bend right] node[midway,sloped,above]{\scriptsize $1/2$} (min3);
      \node[min] (min4) [below=of max5, xshift=0.8cm] {\large \bf 4}
        edge [head, bend left] node[left,\colorpayment]{$-2$} (max6)
        edge [tail, bend right] (max6)
        edge [head, bend right] node[midway,sloped,below]{\scriptsize $1/2$} (n3);
      \node[nature] (n4) [below=of max4] {}
        edge [tail] node[midway,below]{\scriptsize $1/2$} (min2)
        edge [tail] node[midway,below]{\scriptsize $1/2$} (min4)
        edge [head] node[above,xshift=-0.2cm,\colorpayment]{$-1$} (max5);
    \end{tikzpicture}
  \end{center}
  \caption{}
  \label{fig:ExampleFullGraph}
\end{figure}

The Shapley operator of this game is
\begin{equation*}
  T(x) =
  \begin{pmatrix}
    {\color{\colorpayment} 2} + x_1 \, \wedge \, {\color{\colorpayment} 1} + \frac{1}{2} (x_1 + x_2)\\
    \big( {\color{\colorpayment} -2} + \frac{1}{2} (x_1 + x_2) \vee {\color{\colorpayment} -1} + x_1 \big) \, \wedge \, {\color{\colorpayment} 2} + \frac{1}{2} (x_1 + x_3)\\
    {\color{\colorpayment} -3} + \frac{1}{2} (x_1 + x_3) \, \vee \, {\color{\colorpayment} -1} + \frac{1}{2} (x_2 + x_4)\\
    {\color{\colorpayment} -2} + x_4 \, \vee \, {\color{\colorpayment} 2} + \frac{1}{2} (x_3 + x_4)
  \end{pmatrix} \enspace .
\end{equation*}
It can be shown that $T$ verifies the ergodic equation~\eqref{eq:Ergodic}
with ergodic constant $\lambda = 1/3$ and $u=(4/3,0,2/3,4)^\mathsf{T}$.
Let us check whether this game is ergodic,
or equivalently, whether the recession function of $T$, denoted by $F$ and given by
\begin{equation}
  \label{eq:RunningEx}
  F(x) =
  \begin{pmatrix}
    x_1 \, \wedge \, \frac{1}{2} (x_1 + x_2)\\
    \big(\frac{1}{2} (x_1 + x_2) \vee x_1 \big) \, \wedge \, \frac{1}{2} (x_1 + x_3)\\
    \frac{1}{2} (x_1 + x_3) \, \vee \, \frac{1}{2} (x_2 + x_4)\\
    x_4 \, \vee \, \frac{1}{2} (x_3 + x_4)
  \end{pmatrix} \enspace ,
\end{equation}
has only trivial fixed points.

To answer these questions, we need to construct the Galois connection induced by the game.
Firstly, we check that
\begin{align*}
  \F^- &= \big\{ \emptyset, \{1\}, \{1,2\}, \{1,2,3,4\} \big\} \enspace ,\\
  \F^+ &= \big\{ \emptyset, \{4\}, \{1,2,3,4\} \big\} \enspace .
\end{align*}
This can be seen on the graph represented in Figure~\ref{fig:ExampleFullGraph}.
Indeed, following the game-theoretic interpretation (Proposition~\ref{prop:GameInterpretation}), we observe that player \MIN\ can always make sure that the state remains in $\{1\}$ or in $\{1,2\}$, and that player \MAX\ can always make sure that it stays in $\{4\}$.
Alternatively, we can construct the Boolean abstractions of $F$, namely
\begin{equation*}
  F^+(x) =
  \begin{pmatrix}
    x_1\\
    x_1 \, \vee \, (x_2 \wedge x_3)\\
    x_1 \vee x_3 \vee x_2 \vee x_4\\
    x_3 \vee x_4
  \end{pmatrix} \quad \text{and} \quad
  F^-(x) =
  \begin{pmatrix}
    x_1 \wedge x_2\\
    x_1 \wedge x_3\\
    (x_1 \wedge x_3) \vee (x_2 \wedge x_4)\\
    x_4
  \end{pmatrix} \enspace ,
\end{equation*}
and check that
\[
F^+(\indic_{\{2,3,4\}}) \leq \indic_{\{2,3,4\}}, \quad F^+(\indic_{\{3,4\}}) \leq \indic_{\{3,4\}} \quad \text{and} \quad F^-(\indic_{\{4\}}) \geq \indic_{\{4\}} \enspace .
\]

By definition of the Galois connection, or using its characterization by the Boolean operators, we get that
\begin{align*}
  &\Phi(\{1\}) = \Phi(\{1,2\}) = \{4\} \enspace ,\\
  &\Phi^\star(\{4\}) = \{1,2\} \enspace .
\end{align*}

We can thus conclude by Theorem~\ref{thm:Galois} and Theorem~\ref{th-game-ergodicity-full} that the game is not ergodic. 

\subsection{Finding a fixed point with prescribed argmin}

We now address the problem of finding fixed points of $F$ with fixed $\argmin$.
Since
$\indic_{\{2,3,4\}}$ and 
$\indic_{\{3,4\}}$ are the only nontrivial fixed points of $F^+$,
we know from Lemma~\ref{lem:NecessaryCondition} that $\{1\}$ and $\{1,2\}$ are the only possible candidates for nontrivial $\argmin$.

The set $\{1,2\}$ is closed with respect to the Galois connection.
Thus, according to Lemma~\ref{lem:ArgmaxPhi}, $F$ has a fixed point whose $\argmin$ is $\{1,2\}$.
Moreover, its $\argmax$ can only be $\{4\}$.
We can check that the vector $(0,0,1/2,1)^\mathsf{T}$ is a fixed point with these properties.
	
As for the set $\{1\}$, we cannot conclude directly from Lemma~\ref{lem:EmptyPhi} or Lemma~\ref{lem:ArgmaxPhi}.
According to Theorem~\ref{thm:Reduction}, we need to construct a reduced operator, $F^\vartriangle$, defined on $\R^{\{1,2\}}$ ($\{1,2\}$ being the closure of $\{1\}$):
\[
F^\vartriangle(x) =
\begin{pmatrix}
  x_1 \, \wedge \, \frac{1}{2} (x_1 + x_2)\\
  x_1 \, \vee \, \frac{1}{2} (x_1 + x_2)
\end{pmatrix} \enspace .
\]
The directed graph associated with this operator is represented in Figure~\ref{fig:ExampleReducedGraph}.
\begin{figure}[htbp]
  \begin{center}
    \begin{tikzpicture}
      [node distance=1.6cm, on grid, bend angle=20, auto,
        min/.style={circle,draw,minimum size=7mm,fill=black!15},
        max/.style={rectangle,draw,minimum size=5mm},
        nature/.style={diamond,draw,inner sep=0pt,minimum size=6mm},
        tail/.style={->,shorten >=3pt,>=angle 60}, head/.style={<-,shorten <=3pt,>=angle 60}]
      \node[min] (min1) at (0,0) {\large \bf 1};
      \node[max] (max1) [left=of min1, xshift=-0.8cm] {}
        edge [head] (min1);
      \node[nature] (n1) [below=of max1] {}
        edge [tail] node[midway,sloped,above]{\scriptsize $1/2$} (min1)
        edge [head] (max1);
      \node[max] (max2) [right=of n1] {}
        edge [tail] (n1)
        edge [tail] (min1);
      \node[min] (min2) [below=of max2, xshift=0.8cm] {\large \bf 2}
        edge [head, bend left] node[midway,sloped,below]{\scriptsize $1/2$} (n1)
        edge [tail] (max2);
      \node[max] (max3) [right=of max2] {}
        edge [head, bend left] (min1)
        edge [tail, bend right]  (min1);
    \end{tikzpicture}
  \end{center}
  \caption{}
  \label{fig:ExampleReducedGraph}
\end{figure}

We check that for this reduced operator we have
\begin{align*}
  \F^- &= \big\{ \emptyset, \{1\}, \{1,2\} \big\} \enspace ,\\
  \F^+ &= \big\{ \emptyset, \{1,2\} \big\} \enspace .
\end{align*}
Hence, $\Phi(\{1\}) = \emptyset$ and by Lemma~\ref{lem:EmptyPhi}, we know that $F^\vartriangle$ has no fixed point whose $\argmin$ is $\{1\}$.
According to Theorem~\ref{thm:Reduction}, the same holds for $F$.

We conclude that any nontrivial fixed point $u$ of $F$ verifies $u_1 = u_2 < u_3 < u_4$.
Furthermore, from~\eqref{eq:RunningEx} we readily get that $u_3 = \frac{1}{2}(u_2+u_4)$.
These conditions are also sufficient for a point to be a nontrivial fixed point of $F$.
As a consequence, assuming that in Figure~\ref{fig:ExampleFullGraph} the value of the payments can change, all the realizable mean payoff vectors $\chi$ are characterized by
\[ \chi_1 = \chi_2 \leq \chi_4, \quad \chi_3 = \frac{1}{2} (\chi_1 + \chi_4) \enspace . \]

\section{Summary and discussion of the main results}

It is convenient to give here a synthetic description of our results.
We use the notations and definitions of the previous sections, 
in particular the definition of the Galois connection and of the hypergraph
given in Section~\ref{galois-section} and \ref{compact-section}
respectively, and 
and the definition of Assumption~\ref{StandardAssumptions} given in
Section~\ref{compact-section}.
Combining Theorems~\ref{th-game-ergodicity-full},  \ref{thm:Galois} and \ref{th-conjugates}, and
Corollary~\ref{coro:StructuralFixedPoints},
we obtain the following result, which shows 
that most of the classical characterizations of ergodicity for finite state Markov chains, listed in Theorem~\ref{thm:ErgodicityMarkovChain}, carry over to the two-player case, up to the essential discrepancy that the directed graph of the transition probability matrix is now replaced by a pair of directed hypergraphs depending on the transition probability.

\begin{theorem}[Ergodicity of zero-sum games]
  \label{th-game-ergodicity}
  Let us fix a state space $S=[n]$, and the nonempty actions
  spaces $A_i$ and $B_i$ of the two players.
  Let $r$ be a bounded transition payment, let $P$ be a transition probability,
  and let  $T=T(r,P)$ be the Shapley operator of the game $\Gamma(r,P)$.
  Then, the following properties are equivalent:
  \begin{enumerate}
    \item\label{game-ergodic-1} the recession function $\widehat{T}=T(0,P)$ has only trivial fixed points;
    \item\label{game-ergodic-2} the mean payoff vector of the game $\Gamma(r+g,P)$ does exist and is constant for all additive perturbations $g$ of the transition payment, depending only of the state (so $g_i^{a,b}=g_i$, for all $i\in S$, $a\in A_i$ and $b\in B_i$);
    \item\label{game-ergodic-3} the ergodic equation $g + T(u)=\lambda \unit + u$ is solvable 
      for all vectors $g\in \R^n$;
      \item\label{game-ergodic-4} there does not exist a pair of conjugate subsets of states with respect to the Galois connection $(\Phi,\Phi^\star)$ associated with the recession function $\widehat{T}=T(0,P)$.
  \end{enumerate}
  Assume in addition that Assumption~\ref{StandardAssumptions} holds.
  Then, the preceding conditions are equivalent to the following one:
  \begin{enumerate}
    \setcounter{enumi}{4}
    \item \label{eq-hypergraph}
    there does not exist a pair of conjugate subsets of states
    with respect to the hypergraphs $(\bar{G}^+,\bar{G}^-)$ associated with the transition probability $P$.
  \end{enumerate}
In particular (still making Assumption~\ref{StandardAssumptions}), 
the ergodicity property of a game $\Gamma(r,P)$ only depends on the support of $P$.
\end{theorem}

The classical theory of additive functionals of the trajectory of Markov chains corresponds to the zero-player case of zero-sum game theory, or equivalently to the case where each player has only one possible action in each state.
By applying the above theorem to the degenerate Shapley operator
\[ T(x) = g + Px \]
with recession function
\[ \hat{T}(x)  = P x \enspace , \]
the characterizations \pref{it-1}--\pref{prop-ergodic-graph} of the ergodicity of a Markov chain with transition probability matrix $P$,
listed in Theorem~\ref{thm:ErgodicityMarkovChain}, are readily recovered,
see in particular Remark~\ref{remark-hypergraph-graph} for the 
characterization~\pref{prop-ergodic-graph}
(we exclude the characterization of Point~\pref{it-last}
of Theorem~\ref{thm:ErgodicityMarkovChain}, in terms of the uniqueness of the invariant measure, which has no nonlinear analogue).

When the action spaces are finite, an algorithmic issue is to check ergodicity.
We noted that a result of Yang and Zhao~\cite{YZ04} implies
that checking the non-ergodicity is NP-hard, and proved that this problem
is NP-complete (Corollary~\ref{thm:NonTrivialFP}) but fixed parameter tractable (Theorem~\ref{th-fptrac}).

As a refinement of the present ergodicity results we have considered the problem of characterizing the fixed point set $\mathcal{W} := \{ w \in \R^n \mid F(w)=w \}$ of a payment-free Shapley operator $F$.
This problem has been well studied in the one-player case.
In particular, when the action spaces are finite,
$\mathcal{W}$ is known to be sup-norm isometric
to a polyhedral cone with 
a well characterized dimension~\cite{AG03}.
In the two-player case, the properties of the fixed point set
$\mathcal{W}$ are less understood.
In order to get information on this set, we have considered in particular the 
problem of the existence of a fixed point $w$ of $F$ such that $w_i$ is minimal precisely when $i$ belongs to a prescribed subset $I\subset S$.
We showed in Theorem~\ref{thm:I=Min} that this problem can be solved in polynomial time, by  Algorithm~\ref{algo1}. 
We also showed that we can check whether $F$ has a fixed
point with prescribed argmin and argmax in polynomial time. 

Such results deal with 
the ``order abstraction'' of the fixed point set 
of $F$. A natural
refinement would be to ask whether, for a given partition
$I_1\cup\dots\cup I_k$ of the state space $S$, 
there is a fixed point $w$ of $F$ such that
\[
w_i = w_j ,\; \forall i,j\in I_m, \forall m\in[k], \quad\text{and}\quad
w_{i_1}<w_{i_2}<\dots<w_{i_k},\; \forall i_1 \in I_1,\dots,i_k\in I_k \enspace .
\]
We do not know whether this can be checked in polynomial time
for any $k\geq 3$.

\def\cprime{$'$}
\providecommand{\href}[2]{#2}
\providecommand{\arxiv}[1]{\href{http://arxiv.org/abs/#1}{arXiv:#1}}
\providecommand{\url}[1]{\texttt{#1}}
\providecommand{\urlprefix}{URL }

\end{document}